\documentclass{SCAEJOL}
\numberwithin{equation}{section}
\begin{document}

\Year{2015} %
\Month{January}
\Vol{58} %
\No{1} %
\BeginPage{1} %
\EndPage{28} %
\AuthorMark{Guan Q A {\it et al.}}
\ReceivedDay{November 4, 2014}
\AcceptedDay{November 6, 2014}
\PublishedOnlineDay{; published online November 28, 2014}
\DOI{10.1007/s11425-014-4946-4} 

\title{Optimal constant in an ${\bm L}^{\bm2}$ extension problem\\ and a proof of a conjecture
of Ohsawa}{}


\author[1,2]{GUAN Qi'An}{}
\author[3,4]{ZHOU XiangYu}{Corresponding author}

\address[{\rm1}]{Beijing International Center for Mathematical Research, Peking University, Beijing {\rm 100871}, China;}
\address[{\rm2}]{School of Mathematical Sciences, Peking University, Beijing {\rm 100871}, China;}
\address[{\rm3}]{Institute of Mathematics, Academy of Mathematics and Systems Science,\\ Chinese Academy of Sciences, Beijing {\rm 100190}, China;}
\address[{\rm4}]{Hua Loo-Keng Key Laboratory of Mathematics, Chinese Academy of Sciences, Beijing {\rm 100190}, China}
\Emails{guanqian@amss.ac.cn,
xyzhou@math.ac.cn}\maketitle


 {\begin{center}
\parbox{14.5cm}{\begin{abstract}
In this paper, we solve the optimal constant problem in the
setting of Ohsawa's generalized $L^{2}$ extension theorem. As
applications, we prove a conjecture of Ohsawa and the extended Suita
conjecture, we also establish some relations between Bergman kernel
and logarithmic capacity on compact and open Riemann surfaces.\vspace{-3mm}
\end{abstract}}\end{center}}



\renewcommand{\baselinestretch}{1.2}
\begin{center} \renewcommand{\arraystretch}{1.5}
{\begin{tabular}{lp{0.8\textwidth}} \hline \scriptsize
{\bf Citation:}\!\!\!\!&\scriptsize Guan Q A, Zhou X Y. Optimal constant in an $L^2$ extension problem and a proof of a conjecture
of Ohsawa. Sci China Math, 2015, 58, doi: 10.1007/s11425-014-4946-4\vspace{1mm}
\\
\hline
\end{tabular}}\end{center}

\baselineskip 11pt\parindent=10.8pt  \wuhao

\section{Introduction and main results}
\vspace{-1mm}
\subsection{Introduction and background}
Cartan's theorems A and B on Stein manifolds are important and
fundamental in several complex variables. An equivalent version is
an extension theorem as follows. Given a holomorphic vector bundle
on a Stein manifold and a closed complex subvariety in the Stein
manifold, then any section of the restriction of the bundle to the
complex subvariety can be holomorphically extended to the section of
the bundle over the Stein manifold. A natural question is that if
the holomorphic section of the bundle on the complex subvariety is
of a special property (say, invariant w.r.t. a group action or
bounded or $L^2$), could the holomorphic extension be still of the
special property? For the case of invariant version of the extension
theorem, the reader is referred to \cite{zhou2002}. In the present
paper, we deal with the question for the case of $L^2$ extension.

\subsubsection{\it Optimal constant problem in the $L^2$ extension theorem for hypersurfaces}
Ohsawa and Takegoshi \cite{ohsawa-takegoshi} obtained the famous
$L^2$ extension theorem without negligible weight for the
hypersurface with a holomorphic defining function. In the setting of
Ohsawa and Takegoshi, the optimal constant problem in the $L^2$
extension theorem was widely discussed for various cases by
many authors such as Berndtsson, Blocki, Demailly, Guan-Zhou-Zhu, Manivel,
McNeal-Varolin, Siu etc., in \cite{ohsawa-takegoshi, manivel93,
ohsawa2, siu96, berndtsson, demailly99, siu00, mcneal-varolin,
guan-zhou-zhu11, guan-zhou-zhu10, blocki12, blocki12a}. After his
joint paper with Takegoshi, Ohsawa in \cite{ohsawa3} obtained the
$L^2$ extension theorem on pseudoconvex domains in $\mathbb{C}^{n}$
with negligible weights for hypersurface with a holomorphic defining
function. In the setting of Ohsawa in \cite{ohsawa3}, continuing our
work \cite{guan-zhou-zhu10}, we in \cite{guan-zhou12} obtained the
optimal constant version of the $L^2$ extension theorem with
negligible weights on Stein manifolds as follows.

\begin{theorem}[\!\!\cite{guan-zhou12}]\label{t:guan-zhou}\quad
Let $X$ be a Stein manifold of dimension $n$. Let $\varphi$ and
$\psi$ be plurisubharmonic functions on $X$. Assume that $w$ is a
holomorphic function on $X$ such that
$\sup_X(\psi+2\log|w|)\leq0$ and $dw$ does not vanish
identically on any branch of $w^{-1}(0)$. Denote $H=w^{-1}(0)$
and $H_0=\{x\in H:dw(x)\neq 0\}$. Then  there exists a uniform
constant $C=1$ such that, for any holomorphic $(n-1)$-form $f$ on $H_0$
satisfying
$$c_{n-1}\int_{H_0}{\rm e}^{-\varphi-\psi}f\wedge\bar
f<\infty,$$
where $c_{k}=(-1)^{\frac{k(k-1)}{2}}\sqrt{-1}^k$ for
$k\in\mathbb{Z}$, there exists a holomorphic $n$-form F on $X$
satisfying $F=dw\wedge \tilde{f}$ on $H_0$ with
$\imath^*\tilde{f}=f$ and
$$c_n\int_{X}{\rm e}^{-\varphi}F\wedge\bar F\leq 2C\pi c_{n-1}
\int_{H_0}{\rm e}^{-\varphi-\psi}f\wedge\bar f,$$
where $\imath:H_0\longrightarrow X$ is the inclusion map.
\end{theorem}

Our proof in \cite{guan-zhou12} was based on the methods of
\cite{guan-zhou-zhu11} and \cite{guan-zhou-zhu10}. In the present
paper, we will consider some generalization of the above theorem.

\subsubsection{\it Suita conjecture}
Suita conjecture (see\cite{suita}), which was posed originally on
open Riemann surfaces in 1972, was motivated to answer a question
posed by Sario and Oikawa in \cite{sario} about the relation between
the Bergman kernel $\kappa_{\Omega}$ for holomorphic $(1,0)$ forms
on an open Riemann surface $\Omega$ and logarithmic capacity
$c_{\beta}(z)$ which is locally defined by
$c_{\beta}(z)=\exp\lim_{\xi\rightarrow
z}(G_{\Omega}(\xi,z)-\log|\xi-z|)$ on $\Omega$, which admits a Green
function $G_{\Omega}$. The conjecture is stated below.

\th{Suita conjecture.}
{\it $(c_{\beta}(z))^{2}|dz|^{2}\leq\pi \kappa_{\Omega}(z)$, for any $z\in\Omega$.}\vspace{1mm}

The conjecture is true (see \cite{blocki12a,guan-zhou12}).

\subsubsection{\it Ohsawa's generalized $L^2$ extension theorem}
In \cite{ohsawa5}, Ohsawa considered an even more general setting
than before in \cite{ohsawa-takegoshi,ohsawa2,ohsawa3,ohsawa4} and
proved a general extension theorem (main theorem in \cite{ohsawa5}),
which covers earlier main results in
\cite{ohsawa-takegoshi,ohsawa2,ohsawa3,ohsawa4}:

\begin{theorem}[\!\!\cite{ohsawa5}]\label{t:ohsawa5}\quad
Let $(M,S)$ satisfy condition $({\rm ab})$, $h$ be a smooth metric on holomorphic vector
bundle $E$ on $M$ with
rank $r$. Then, for any function
$\Psi$ on $ M$ such that $\Psi\in \Delta_{h,\delta}(S)\cap C^{\infty}(M\setminus S)$,
there exists a uniform constant
$C=\max_{1\leq k\leq n}\frac{2^{8}\pi}{\frac{\pi^{k}}{k!}}$ such that, for
any holomorphic section $f$ of $K_{M}\otimes E|_{S}$ on $S$ satisfying
\[\sum_{k=1}^{n}\frac{\pi^{k}}{k!}\int_{S_{n-k}}|f|^{2}_{h}dV_{M}[\Psi]<\infty,\] there
exists a holomorphic section $F$ of $K_{M}\otimes E$ on $M$ satisfying $F = f$ on $ S$ and
\begin{eqnarray*}
\int_{M}|F|^{2}_{h}dV_{M}
\leq C(1+\delta^{-\frac{3}{2}})\sum_{k=1}^{n}\frac{\pi^{k}}{k!}
\int_{S_{n-k}}|f|^{2}_{h}dV_{M}[\Psi].
\end{eqnarray*}
Especially, if $\Psi\in \Delta(S)\cap\Delta_{h,\delta}(S)\cap C^{\infty}(M\setminus S)$,
then there exists a holomorphic section $F$ of $K_{M}\otimes E$ on $M$ satisfying $F = f$
on $ S$ and
\begin{eqnarray*}
\int_{M}|F|^{2}_{h}dV_{M}
\leq C\sum_{k=1}^{n}\frac{\pi^{k}}{k!}\int_{S_{n-k}}|f|^{2}_{h}dV_{M}[\Psi].
\end{eqnarray*}
\end{theorem}

For the sake of completeness, let us recall and explain the symbols
and notation in the above theorem as in \cite{ohsawa5}.

Let $M$ be an $n$-dimensional complex manifold, and $S$ be a closed
complex submanifold of $M$. Let $dV_{M}$ be a continuous volume form
on $M$. We consider a class of upper-semi-continuous function $\Psi$
from $M$ to the interval $[-\infty,0)$ such that

$(1)$ $\Psi^{-1}(-\infty)\supset S$, and $\Psi^{-1}(-\infty)$ is a closed subset of $M$;

$(2)$ If $S$ is $l$-dimensional around a point $x$, there exists a
local coordinate $(z_{1},\ldots,z_{n})$ on a neighborhood $U$ of $x$
such that $z_{l+1}=\cdots=z_{n}=0$ on $S\cap U$ and
$$\sup_{U\setminus S}\bigg|\Psi(z)-(n-l)\log\sum_{l+1}^{n}|z_{j}|^{2}\bigg|<\infty.$$

The set of such polar functions $\Psi$ will be denoted by $\#(S)$.

For each $\Psi\in\#(S)$, one can associate a positive measure $dV_{M}[\Psi]$ on $S$
as the minimum element of the
partial ordered set of positive measures $d\mu$ satisfying
$$\int_{S_{l}}fd\mu\geq\limsup_{t\to\infty}\frac{2(n-l)}{\sigma_{2n-2l-1}}
\int_{M}f{\rm e}^{-\Psi}{1} _{\{-1-t<\Psi<-t\}}dV_{M}$$ for any
nonnegative continuous function $f$ with ${\rm Supp} f\subset\subset M$.
Here $S_{l}$ denotes the $l$-dimensional component of $S$, and
$\sigma_{m}$ denotes the volume of the unit sphere in
$\mathbb{R}^{m+1}$.

Let $\omega$ be a K\"{a}hler metric on $M\setminus (X\cup S)$. We
can also define measure $dV_{\omega}[\Psi]$ on $S\setminus X$ as the
minimum element of the partial ordered set of positive measures
$d\mu'$ satisfying
$$\int_{S_{l}}fd\mu'\geq\limsup_{t\to\infty}\frac{2(n-l)}{\sigma_{2n-2l-1}}
\int_{M\setminus (X\cup S)}
f{\rm e}^{-\Psi}{1}_{\{-1-t<\Psi<-t\}}dV_{\omega}$$ for any
nonnegative continuous function $f$ with ${\rm Supp} (f)\subset\subset
M\setminus X$ (as
${\rm Supp}({1}_{\{-1-t<\Psi<-t\}})\cap {\rm Supp}(f)$ $\subset\subset M\setminus (X\cup S),$
the right-hand side of the above inequality is well-defined).

Let $u$ be a continuous section of $K_{M}\otimes E$, where $E$ is a holomorphic
vector bundle equipped with
a continuous metric $h$ on $M$.
We define
$$|u|^{2}_{h}|_{V}:=\frac{c_{n}h(e,e)v\wedge\bar{v}}{dV_{M}},$$
where $u|_{V}=v\otimes e$ for an open set $V\subset M\setminus X$, $v$ is a continuous
section of $K_{M}|_{V}$
and $e$ is a continuous section of $E|_{V}$ (especially, we define
$$|u|^{2}|_{V}:=\frac{c_{n}u\wedge\bar{u}}{dV_{M}},$$
when $u$ is a continuous section of $K_{M}$). It is clear that
$|u|^{2}_{h}$ is independent of the choice of $V$. Actually by the
following remark, one may see the following relationship between
$dV_{\omega}[\Psi]$ and $dV_{M}[\Psi]$ (resp. $dV_{\omega}$ and
$dV_{M}$), precisely,
\begin{align}
\label{equ:9.1}
&\int_{M\setminus X}f|u|^{2}_{h,\omega}dV_{\omega}[\Psi]
=\int_{M\setminus X}f|u|^{2}_{h}dV_{M}[\Psi],\\
\label{equ:9.2} &\bigg({\rm resp.}\ \int_{M\setminus
X}f|u|^{2}_{h,\omega}dV_{\omega} =\int_{M\setminus
X}f|u|^{2}_{h}dV_{M}\bigg),
\end{align}
where $f$ is a continuous function with compact support on $M\setminus X$.

\begin{remark}\quad
For the neighborhood $U$, let $u|_{U}=v\otimes e$. Note that
\begin{align}
\label{equ:9.3}
\int_{M\setminus X}f{1}_{\{-1-t<\Psi<-t\}}
|u|^{2}_{h,\omega}{\rm e}^{-\Psi}dV_{\omega}
&=\int_{M\setminus X}f{1}_{\{-1-t<\Psi<-t\}}h(e,e)c_{n}v\wedge\bar{v}{\rm e}^{-\Psi}\nonumber\\
&=\int_{M\setminus X}f{1}_{\{-1-t<\Psi<-t\}}|u|^{2}_{h}{\rm e}^{-\Psi}dV_{M},
\end{align}
and respectively,
\begin{equation}
\int_{M\setminus (X\cup S)}f
|u|^{2}_{h,\omega}{\rm e}^{-\Psi}dV_{\omega}
=\int_{M\setminus (X\cup S)}f h(e,e)c_{n}v\wedge\bar{v}{\rm e}^{-\Psi}
=\int_{M\setminus (X\cup S)}f |u|^{2}_{h}{\rm e}^{-\Psi}dV_{M},
\end{equation}
where $f$ is a continuous function with compact support on $M\setminus X$. As
$${\rm Supp}({1}_{\{-1-t<\Psi<-t\}})\cap {\rm Supp}(f)\subset\subset M\setminus (X\cup S),$$
equality (\ref{equ:9.3}) is well-defined.
Then we have equalities (\ref{equ:9.1}) and (\ref{equ:9.2}).
\end{remark}

It is clear that $|u|^{2}_{h}$ is independent of the choice of $V$,
while $|u|^{2}_{h}dV_{M}$ is independent of the choice of $dV_{M}$
(resp. $|u|^{2}_{h}dV_{M}[\Psi]$ is
independent of the choice of $dV_{M}$).
Then the space of $L^{2}$ integrable holomorphic section of $K_{M}$ is denoted by
$A^{2}(M,K_{M},dV_{M}^{-1},dV_{M})$ (resp.
the space of holomorphic section of $K_{M}|_{S}$ which is $L^{2}$ integrable with
respect to the measure
$dV_{M}[\Psi]$ is denoted by $A^{2}(S,K_{M}|_{S},dV_{M}^{-1},$ $dV_{M}[\Psi])$).

\begin{definition}\quad
Let $M$ be an $n$-dimensional complex manifold with a continuous
volume form $dV_{M}$, and $S$ be a closed complex submanifold of
$M$. We call the data $(M,S)$ satisfy the condition $({\rm ab})$ if $M$ and
$S$ satisfy the following conditions:

There exists a closed subset $X\subset M$ such that:

$({\rm a})$ $X$ is locally negligible with respect to $L^2$ holomorphic functions, i.e.,
for any local coordinate
neighborhood $U\subset M$ and for any $L^2$ holomorphic function $f$ on $U\setminus X$,
there exists an $L^2$
holomorphic function  $\tilde{f}$ on $U$ such that $\tilde{f}|_{U\setminus X}=f$ with
the same $L^{2}$ norm.

$({\rm b})$ $M\setminus X$ is a Stein manifold which intersects with every
component of $S$.
\end{definition}

\begin{remark}\quad
In fact, the condition (ab) is the same as condition (1) in Theorem
4 in \cite{ohsawa5}. The data $(M,S)$ with the condition (ab) includes all the following well-known examples:

(1) $M$ is a Stein manifold (including open Riemann surfaces), and $S$ is
any closed complex submanifold of $M$;

(2) $M$ is a complex projective algebraic manifold (including
compact Riemann surfaces), and $S$ is any closed complex submanifold
of $M$;

(3) $M$ is a projective family (see \cite{siu00}), and $S$ is any closed
complex submanifold of $M$.
\end{remark}

The following remark shows the extension properties of holomorphic sections
of holomorphic vector bundles from $M\setminus X$ to $M$.

\begin{remark}
\label{r:extend}\quad
Let $(M,S)$ satisfy the condition (ab),
$h$ be a singular metric on holomorphic line bundle $L$ on $M$ (resp. continuous
metric on holomorphic vector
bundle $E$ on $M$ with rank $r$),
where $h$ has locally positive lower bound.
Let $F$ be a holomorphic section of $K_{M\setminus X}\otimes E|_{M\setminus X}$,
which satisfies $\int_{M\setminus X}|F|^{2}_{h}<\infty$.

As $h$ has locally positive lower bound and $M$ satisfies (a) of the
condition (ab), there is a holomorphic section $\tilde{F}$ of
$K_{M}\otimes L$ on $M$ (resp. $K_{M}\otimes E$), such that
$\tilde{F}|_{M\setminus X}=F$.
\end{remark}

Let $\Delta_{h,\delta}(S)$ be the subset of function $\Psi$ in $\#(S)$ which satisfies
$\Theta_{h{\rm e}^{-\Psi}}\geq 0$
and $\Theta_{h{\rm e}^{-(1+\delta)\Psi}}\geq 0$ on $M\setminus S$ in the sense of Nakano.

Let $\Delta(S)$ be the subset of plurisubharmonic functions $\Psi$
in $\#(S)$, $\varphi$ be a locally integrable function on $M$. Let
$\Delta_{\varphi,\delta}(S)$ be the subset of functions $\Psi$ in
$\#(S)$, such that $\Psi+\varphi$ and $(1+\delta)\Psi+\varphi$ are
both plurisubharmonic functions on $M$.\emph{}

\subsection{Main results}
In the present paper, we give some generalizations of both Theorems \ref{t:guan-zhou}
and \ref{t:ohsawa5}.

\subsubsection{\it Optimal constant in the generalized $L^2$ extension theorem for
non-smooth polar function}
In the following theorem, we give an optimal constant of a
generalization of Theorem \ref{t:ohsawa5} for trivial line bundle
and  non-plurisubharmonic polar function $\Psi$. After that, we show
that the constant $C$, which is equal to $1$, is optimal.

\begin{theorem}[Main theorem 1]\label{t:guan-zhou3}\quad
Let $(M,S)$ satisfy condition $({\rm ab})$, $\varphi$ be a continuous
 function on $M$. Then, for negative function
$\Psi$ on $M$ satisfying $\Psi\in \Delta_{\varphi,\delta}(S)$, there exists a uniform
constant
$C=1$, such that, for
any holomorphic section $f$ of $K_{M}|_{S}$ on $S$ satisfying
\[\sum_{k=1}^{n}\frac{\pi^{k}}{k!}\int_{S_{n-k}}
|f|^{2}{\rm e}^{-\varphi}dV_{M}[\Psi]<\infty,\] there
exists a holomorphic section $F$ of $K_{M}$ on $M$ satisfying $F = f$ on $ S$ and
\begin{eqnarray*}
\int_{M}|F|^{2}{\rm e}^{-\varphi}dV_{M}
\leq C(1+\delta^{-1})\sum_{k=1}^{n}\frac{\pi^{k}}{k!}
\int_{S_{n-k}}|f|^{2}{\rm e}^{-\varphi}dV_{M}[\Psi],
\end{eqnarray*}
where $|f|^{2}:=\frac{c_{n}f\wedge\bar{f}}{dV_{M}}$.
\end{theorem}

\begin{remark}\label{r:guan-zhou3} \quad
We will see that the above constant is optimal
in the proof. Especially, we will illustrate that for any given
planar domain $\Omega$ in $\mathbb{C}$ and point $z_{0}\in\Omega$,
the above constant is also optimal, where $S=\{z_{0}\}$.
\end{remark}

Note that for any holomorphic line bundle $L$ on Stein manifold
$M\setminus X$, we can choose a complex hypersurface $H$ on
$M\setminus X$, such that $L|_{M\setminus (X\cup H)}$ is a trivial
line bundle, where $M\setminus (X\cup H)$ is a Stein manifold, and
$H$ does not contain any component of $S\setminus X$. By Remark
\ref{r:extend}, we obtain the following corollary:

\begin{corollary}\label{coro:line1}\quad
Let $(M,S)$ satisfy condition $({\rm ab})$, $L$ be a holomorphic line bundle on $M$ with a
continuous metric $h$
$($resp. a singular metric $h$ satisfying $\Theta_{h}\geq \omega$, where $\omega$ is a
smooth real $(1,1)$-form
on $M)$. Then, for negative function
$\Psi\in\#(S)$ on $M$ satisfying $\Theta_{h{\rm e}^{-\Psi}}\geq0$ and
$\Theta_{h{\rm e}^{-(1+\delta)\Psi}}\geq0$
$($resp. $\sqrt{-1}\partial\bar\partial\Psi+\omega>0$ and
$(1+\delta)\sqrt{-1}\partial\bar\partial\Psi+\omega>0)$
in the sense of current on $M$, there exists a uniform constant
$C=1$, such that, for
any holomorphic section $f$ of $K_{M}\otimes L|_{S}$ on $S$ satisfying
\[\sum_{k=1}^{n}\frac{\pi^{k}}{k!}\int_{S_{n-k}}|f|^{2}_{h}dV_{M}[\Psi]<\infty,\] there
exists a holomorphic section $F$ of $K_{M}\otimes L$ on $M$ satisfying $F = f$ on $ S$ and
\begin{eqnarray*}
\int_{M}|F|^{2}_{h}dV_{M}
\leq C(1+\delta^{-1})\sum_{k=1}^{n}\frac{\pi^{k}}{k!}
\int_{S_{n-k}}|f|^{2}_{h}dV_{M}[\Psi].
\end{eqnarray*}
\end{corollary}

\subsubsection{\it Some relations between Bergman kernel and logarithmic capacity
on compact Riemann surfaces}
Let X be a compact Riemann surface with genus $g\geq 2$ (resp. $g=1$ complex torus,
which is denoted by
$X_{\tau}:=\mathbb{C}/(\mathbb{Z}+\tau\mathbb{Z})$ ($\tau\in\mathbb{C}$, $\Im\tau>0$)).
Then there exists a
conformal metric $\omega$ on $X$, obtained by descending the
Poincar\'{e} metric from its universal covering space to itself, such that ${\rm curv}\omega=-1$
(resp. the
K\"{a}hler metric $\omega:=\frac{1}{\Im\tau}dz\otimes d\bar{z}$).

Consider the function $g(p,q): X\times X \to [-\infty,0)$, such that for each fixed
$q\in X $:

(a) $\Delta_{\omega}g(\cdot,q)=-1$ on $X\setminus q$
(here $\Delta_{\omega}$ is the Laplacian with respect to the metric $\omega$);

(b) $g(p,q) = \log {\rm dist}_{\omega}(p,q)+O(1)$, as $p\to q$;

(c) $g(p,p) =-\infty$.

The existence and uniqueness of such a function for general compact Riemann surfaces are
proved by Arakelov \cite{Arak}. And define further
$c_{X}(p):=\exp(\lim_{q\to p}(g(p,q)-\log {\rm dist}_{\omega}(p,q)))$
to be
the logarithmic capacity with respect to $p$.

By Corollary \ref{coro:line1}, we can obtain the relations between Bergman kernel
and logarithmic capacity on
Riemann surface with genus $g\geq 2$  as follows:

\begin{theorem}\label{t:arak}\quad
Let $\kappa_{X,m}$ be the Bergman kernel of holomorphic line bundle $mK_{X}$ $(m\geq2)$,
where $X$ is a
compact Riemann surface with genus $g\geq2$.
Then we have
\begin{equation}
\pi\bigg(1+\frac{1}{(g-1)(m-1)-1}\bigg)|\kappa_{X,m}(p,p)|_{h^{m}}\geq c_{X}^{2}(p),
\end{equation}
where $h$ is the Hermitian metric on $K_{X}$ induced by Poincar\'{e} metric on $X$ $($see the
proof of the present theorem$)$.
\end{theorem}

As Arakelov metric $ds_{A}^{2}:=(\exp(2\lim_{\xi\to
z}(g(\xi,z)-\log|\xi-z|)))|dz|^{2}$ (see \cite{went08}), it is clear that
$c_{X}^{2}=\frac{ds_{A}^{2}}{\omega}$. Note that $\kappa_{X,m}$ is the
maximum value of holomorphic sections on $mK_{X}$ whose $L^{2}$ norm
is $1$ with respect to $\omega$. Then the above theorem reveals a
relation (only depends on $n$ and $g$) between two important objects
appearing in string perturbation theory: Arakelov metric and
holomorphic sections on $mK_{X}$, where Arakelov metric plays an
important role in bosonization (see \cite{hoker88}), and holomorphic
section on $mK_{X}$ corresponds to the zero-mode in string
perturbation theory (see \cite{hoker88}).

By Corollary \ref{coro:line1}, we also can obtain a relation between
Bergman kernel and logarithmic capacity on compact Riemann surface
with genus $g=1$ as follows:

\begin{theorem}\label{t:arak1}\quad
Let $\kappa_{X,d}$ be the Bergman kernel of line bundle $K_{X}\otimes L$, where $X$
is a compact Riemann
surface with genus $g=1$, and $L$ is positive line bundle with degree $d$ $(d>2)$.
Then we have
\begin{equation}
\pi\bigg(1+\frac{1}{\frac{d}{2}-1}\bigg)|\kappa_{X,d}(p,p)|_{\omega,h_{L}}\geq c_{X}^{2}(p),
\end{equation}
where $h_{L}$ is the canonical Hermitian metric on $L$ which satisfies $c_{1}(L)=b\omega$ $(b>0)$.
\end{theorem}

\subsubsection{\it Optimal constant in the generalized $L^2$ extension theorem on line bundles
with singular metric and non-smooth polar function}
In the following theorem, we give a generalization of Theorem \ref{t:ohsawa5} on
canonical line bundles
and  plurisubharmonic polar function $\Psi$ and show that the optimal constant
$C$ is equal to $1$.

\begin{theorem}\label{t:guanzhou1}\quad
Let $(M,S)$ satisfy condition $({\rm ab})$, $\varphi$ be a
plurisubharmonic function on $ M$. Then, for any negative plurisubharmonic function
$\Psi$ on $ M$ such that $\Psi\in \Delta(S)$, there exists a uniform constant
$C=1$, such that, for
any holomorphic section $f$ of $K_{M}|_{S}$ on $S$ satisfying
\[\sum_{k=1}^{n}\frac{\pi^{k}}{k!}\int_{S_{n-k}}|f|^{2}{\rm e}^{-\varphi}dV_{M}[\Psi]<\infty,\]
there
exists a holomorphic section $F$ of $K_{M}$ on $M$ satisfying $F = f$ on $ S$ and
\begin{eqnarray*}
\int_{M}|F|^{2}{\rm e}^{-\varphi}dV_{M}
\leq C\sum_{k=1}^{n}\frac{\pi^{k}}{k!}\int_{S_{n-k}}|f|^{2}{\rm e}^{-\varphi}dV_{M}[\Psi].
\end{eqnarray*}
\end{theorem}

By retraction (see Lemma \ref{l:lem.retra}) and convolution, and Remark \ref{r:extend},
we can obtain the
following corollary:

\begin{corollary}\label{coro:line3}\quad
Let $(M,S)$ satisfy condition $({\rm ab})$, $L$ be a holomorphic line bundle on $M$ with a
singular metric $h$ satisfying $\Theta_{h}\geq0$ in the sense of current. Then, for
negative plurisubharmonic function
$\Psi$ on $M$ satisfying $\Psi\in \Delta(S)$, there exists a uniform constant
$C=1$, such that, for
any holomorphic section $f$ of $K_{M}\otimes L|_{S}$ on $S$ satisfying
\[\sum_{k=1}^{n}\frac{\pi^{k}}{k!}\int_{S_{n-k}}|f|^{2}_{h}dV_{M}[\Psi]<\infty,\] there
exists a holomorphic section $F$ of $K_{M}\otimes L$ on $M$ satisfying $F = f$ on $ S$ and
\begin{eqnarray*}
\int_{M}|F|^{2}_{h}dV_{M}
\leq C\sum_{k=1}^{n}\frac{\pi^{k}}{k!}\int_{S_{n-k}}|f|^{2}_{h}dV_{M}[\Psi].
\end{eqnarray*}
\end{corollary}

\subsubsection{\it A conjecture of Ohsawa}
If $\Delta(S)$ is non-empty, we set $G(z
,S):=(\sup\{u(z):u\in\Delta(S)\})^{*}$, which is the upper envelope
of $\sup\{u(z):u\in\Delta(S)\}$. It is clear that $G(z,S)$ is a
plurisubharmonic function on $M$ (see Choquet's lemma (Lemma 4.23 in
\cite{demailly-book})). By Proposition 9 in \cite{ohsawa5}, we have
$G(z,S)\in\Delta(S)$. $G(z,S)$ is called generalized pluricomplex
Green function on $M$ with poles on $S$. If $\Delta(S)$ is empty,
$G(z ,S):=-\infty$. When $S=\{z\}$ for some $z\in M$, $G(z ,S)$ is
the so-called the pluricomplex Green function (see
\cite{demailly86}).

Let $(M,S)$ satisfy condition (ab),
$G(\cdot ,S)$ be the generalized pluricomplex Green function which is nontrivial.
Let $dV_{M}$ be a continuous volume form on $M$ and let $\{\sigma_{j}\}^{\infty}_{j=1}$
(resp. $\{\tau_{j}\}^{\infty}_{j=1}$) be a complete orthogonal system of $A^{2}
(M,K_{M},dV_{M}^{-1},dV_{M})$
(resp. $A^{2}(S,K_{M}|_{S},dV_{M}^{-1},$ $dV_{M}[G(\cdot,S)])$) and put
$\kappa_{M}=\sum_{j=1}^{\infty}\sigma_{j}\otimes\bar\sigma_{j}\in C^{\omega}
(M,K_{M}\otimes\bar{K}_{M})$
(resp. $\kappa_{M/S}=\sum_{j=1}^{\infty}\tau_{j}\otimes\bar\tau_{j}\in C^{\omega}
(S,K_{M}\otimes\bar{K}_{M})$).

One motivation to estimate the constant $C$ in Theorem
\ref{t:guan-zhou3} comes from the conjecture of Ohsawa (see
\cite{ohsawa5}) on $(M,S)$ satisfying condition (ab) and admitting
nontrivial generalized pluricomplex Green functions on $M$ with
poles on $S$, which is stated below:

\th{Conjecture\ {\rm(Ohsawa)}.}
{\it $(\pi^{k}/k!)\kappa_{M}(x)\geq \kappa_{M/S}(x)$ for any $x\in S_{n-k}$.}\vspace{1mm}

The relationship between the conjecture of Ohsawa and the extension theorem was
observed and explored by Ohsawa \cite{ohsawa5}, and he proved the estimate with
$C=\frac{2^{8}\pi}{\pi^{k}/k!}$. The conjecture of Ohsawa can be seen as
an extension of Suita conjecture (see Section 3 of \cite{ohsawa5}).

Using Theorem \ref{t:guanzhou1}, we get
\begin{corollary}\label{c:ohsawa}\quad
The above Ohsawa's conjecture holds.
\end{corollary}

\subsubsection{\it Extended Suita conjecture}
Given a weight $\rho$, one may define a weighted Bergman space consisting $w$,
such that $\int_{\Omega}\rho w\wedge \bar{w}<+\infty$. Denote its Bergman kernel by $\kappa_{\Omega,\rho}$.

Let $\Omega$ be an open Riemann surface, which
admits a Green function. Let $h$ be a harmonic function on $\Omega$,
and $\rho={\rm e}^{-2h(z)}$. Given a harmonic function, there is an
extended Suita conjecture in \cite{yamada98}:

\th{Extended Suita conjecture.}
{\it $c_{\beta}^{2}(p)|dz|^{2}\leq \pi\rho(p) \kappa_{\Omega,\rho}(p)$ for any $p\in\Omega$.}\vspace{1mm}

Let $z_{0}\in\Omega$, with local coordinate $z$. Let
$p:\Delta\to\Omega$ be the universal covering from unit disc
$\Delta$ to $\Omega$. The fundamental group of $\Omega$ naturally acts on Prym differential,
whose set is denoted by $\Gamma^{\kappa}(\Omega)$. Actually this is equivalent to the following
equivariant version in term of multiplicative Bergman kernel denoted by $\kappa_{\Omega}^{\chi}$,
where $\chi$ is a representation of the fundamental group (see [44]).

\th{Extended Suita conjecture.}
{\it $c_{\beta}^{2}(z_{0})\leq \pi B_{\Omega}^{\chi}(z_{0})$.} \vspace{1mm}

Using Theorem \ref{t:guanzhou1}, we get

\begin{corollary}\label{c:extended_suita}\quad
The above extended Suita conjecture holds.
\end{corollary}

\begin{remark}\label{r:suita}\quad
When $h\equiv0$, the above reduces to Suita conjecture.
\end{remark}

There is an equivalent version of Suita conjecture in terms of
Fuchsian groups as follows:

\begin{remark}\quad
There is a Fuchsian group $\Gamma$ without elliptic elements such
that $\Delta/\Gamma$ is conformally equivalent to $\Omega$ . We can
choose $\Gamma$ such that $0\in \Delta$ corresponds to
$\omega\in\Omega$. When $\Omega$ admits a Green function, the group
$\Gamma$ is of convergence type. In \cite{pom84}, Pommerenke and
Suita proved that Suita conjecture is equivalent to
\begin{equation}
\label{equ:fuch}
\sum_{\gamma\in\Gamma}\gamma'(0)\geq\prod_{\gamma\in\Gamma,\gamma\neq\iota}|\gamma(0)|^{2},
\end{equation}
where $\iota$ denotes the identity.
\end{remark}

\subsubsection{\it Boundary behavior of the quotient of logarithmic capacity and Bergman kernel}
In this subsection, we discuss boundary behavior of the quotient of
logarithmic capacity and Bergman kernel by squeezing function
$s_{\Omega}$ on bounded planar domain $\Omega$ as follows:

\begin{definition}[\!\!\cite{dgz}]\label{def:squeezing function}\quad
Let $\Omega$ be a bounded domain in $\mathbb{C}$. For $p\in \Omega$ and an (open)
holomorphic embedding $f: \Omega\rightarrow \Delta$ with $f(p)=0$, we define
$$s_{\Omega}(p , f)= \sup\{r\mid\Delta(0,r)\subset f(\Omega) \},$$
and the squeezing number $s_\Omega(p)$ of $D$ at $p$ is defined as
$$s_{\Omega}(p)= \sup_f\{s_{\Omega}(p , f) \},$$
where the supremum is taken over all holomorphic embeddings $f: \Omega\rightarrow \Delta$
with $f(p)=0$, $\Delta$ is the  unit ball in $\mathbb{C}$ and $\Delta(0 , r)$ is the ball
in $\mathbb{C}$ with center $0$ and radius $r$. We call $s_\Omega$ the
\emph{squeezing function} on $\Omega$.
\end{definition}

Let $C(\Omega,z):=\frac{c_{\beta}(z)^{2}|dz|^{2}}{\pi\kappa_{\Omega}(z)}$,
where $z\in\Omega$.
We consider the boundary behavior of $C(\Omega,z)$ as follows:

\begin{proposition}\label{t:bound}\quad
Let $\Omega$ be a Riemann surface which is biholomorphic equivalent to
a bounded planar domain.
Then we have $1\geq C(\Omega,z)\geq s_{\Omega}^{2}$.
If $\lim_{z\to\partial\Omega}s_{\Omega}=1$, then
$\lim_{z\to\partial\Omega}C(\Omega,z)=1$.
\end{proposition}

Moveover, when $\Omega$ has smooth boundary, by Theorem 5.2 in \cite{dgz}, we have
$\lim_{z\to\partial\Omega} s_{\Omega}=1.$
By Remark~\ref{r:suita}, we obtain $C(\Omega,z)\leq 1$. Then we have
\begin{corollary}[\!\!\cite{suita}]\label{}\quad
Let $\Omega$ be a Riemann surface which is biholomorphically
equivalent to a bounded planar domain with smooth boundary. Then we
have
$$\lim_{z\to\partial\Omega}C(\Omega,z)=1.$$
\end{corollary}

\subsubsection{\it Optimal constant in the generalized $L^2$ extension theorem on holomorphic
vector bundle and polar function which is smooth outside $S$}
In the following theorem, we give a generalization of Theorem \ref{t:ohsawa5} on
holomorphic vector bundles and polar function $\Psi$ which is smooth outside $S$, and
show that the optimal constant $C$ is equal to $1$.

\begin{theorem}[Main theorem 2]\label{t:guan-zhou-vector}\quad
Let $(M,S)$ satisfy condition $({\rm ab})$, $h$ be a smooth metric on a holomorphic vector
bundle $E$ on $M$ with rank $r$. Then, for any function
$\Psi$ on $ M$ such that $\Psi\in \Delta_{h,\delta}(S)\cap C^{\infty}(M\setminus S)$,
there exists a uniform constant
$C=1$ such that, for
any holomorphic section $f$ of $K_{M}\otimes E|_{S}$ on $S$ satisfying
\[\sum_{k=1}^{n}\frac{\pi^{k}}{k!}\int_{S_{n-k}}|f|^{2}_{h}dV_{M}[\Psi]<\infty,\] there
exists a holomorphic section $F$ of $K_{M}\otimes E$ on $M$ satisfying $F = f$ on $ S$ and
\begin{eqnarray*}
\int_{M}|F|^{2}_{h}dV_{M}
\leq C(1+\delta^{-1})\sum_{k=1}^{n}\frac{\pi^{k}}{k!}
\int_{S_{n-k}}|f|^{2}_{h}dV_{M}[\Psi].
\end{eqnarray*}
\end{theorem}

When $\delta$ approaches to infinity, we have a generalization of
Theorem 4 in \cite{ohsawa5} (the main theorem) of holomorphic vector
bundles in plurisubharmonic case ($\Psi$ is plurisubharmonic on $M$)
and show that the $C$ is equal to $1$, which is optimal.

\begin{corollary}\quad
Let $(M,S)$ satisfy condition $({\rm ab})$,
$h$ be a smooth metric on holomorphic vector bundle $E$ on $M$ with rank $r$.
Then, for any function
$\Psi$ on $ M$ such that $\Psi\in \Delta(S)\cap\Delta_{h,\delta}(S)\cap C^{\infty}
(M\setminus S)$, there exists a uniform constant
$C=1$ such that, for
any holomorphic section $f$ of $K_{M}\otimes E|_{S}$ on $S$ satisfying\[\sum_{k=1}^{n}
\frac{\pi^{k}}{k!}\int_{S_{n-k}}|f|^{2}_{h}dV_{M}[\Psi]<\infty,\] there
exists a holomorphic section $F$ of $K_{M}\otimes E$ on $M$ satisfying $F = f$ on $ S$ and
\begin{eqnarray*}
\int_{M}|F|^{2}_{h}dV_{M}
\leq C\sum_{k=1}^{n}\frac{\pi^{k}}{k!}\int_{S_{n-k}}|f|^{2}_{h}dV_{M}[\Psi].
\end{eqnarray*}
\end{corollary}

\section{Some lemmas used in the proof of main theorems}

In this section, we give some lemmas which will be used in the
proofs of main theorems of the present paper.

\subsection{${\bm L}^{\bm2}$ estimates for $\bar{\bm\partial}$ equations}
In this subsection, we give some lemmas on $L^{2}$ estimates for some
$\bar\partial$ equations, and $\bar\partial^*$ means the Hilbert adjoint operator of
$\bar\partial$.

\begin{lemma}\label{l:lem3}\quad
Let $(X,ds_{X}^{2})$ be a K\"{a}hler manifold of dimension n with a K\"{a}hler
metric $ds_{X}^{2}$,
$\Omega\subset\subset X$ be a domain with $C^\infty$ boundary $b\Omega$,
$\Phi\in C^{\infty}(\overline \Omega)$.
 Let $\rho$ be a $C^{\infty}$ defining function for $\Omega$
such that $|d\rho|=1$ on $b\Omega$.
Let $\eta$ be a smooth function on $\overline{\Omega}$. Then for any $(n,1)$-form
$\alpha={\sum}'_{|I|=n}\alpha_{I\bar j}dz^I\wedge d\bar z^{j}
\in {\rm Dom}_\Omega(\bar{\partial}^*)\cap C^\infty_{(n,1)}(\overline\Omega)$,
\begin{align}
\label{guan1}
&\int_{\Omega}\eta|\bar{\partial}^{*}_{\Phi}\alpha|^{2}{\rm e}^{-\Phi}dV_{X}
+\int_{\Omega}\eta|\bar{\partial}\alpha|^{2}{\rm e}^{-\Phi}dV_{X}\nonumber\\
&\quad=\underset{|J|=1}{{\sum}'}\sum_{i,j=1}^{n}\int_{\Omega}
\eta g^{i\bar j}\overline\nabla_{j}\alpha_{I\bar J}\overline{\overline\nabla_{i}
\alpha^{\bar IJ}}{\rm e}^{-\Phi}dV_{X}
+\sum_{i,j=1}^{n}
\int_{b\Omega}\eta(\partial_i\bar\partial_{j}\rho)\alpha_{I\text{ }}^
{\text{ }i\text{ }}\overline{{\alpha}^{\bar Ij}}{\rm e}^{-\Phi}dS\nonumber
\\&\qquad+\sum_{i,j=1}^{n}\int_{\Omega}
\eta(\partial_i\bar \partial_{j}\Phi)\alpha_{I\text{ }}^{\text{ }i\text{ }}
\overline{{\alpha}^{\bar Ij}}{\rm e}^{-\Phi}dV_{X}
+\sum_{i,j=1}^{n}\int_{\Omega}
-(\partial_i\bar \partial_{j}\eta)\alpha_{I\text{ }}^{\text{ }i\text{ }}
\overline{{\alpha}^{\bar Ij}}{\rm e}^{-\Phi}dV_{X}\nonumber
\\&\qquad+2\mathrm{Re}(\bar\partial^*_\Phi\alpha,\alpha\llcorner(\bar\partial\eta)^\sharp )_
{\Omega,\Phi},
\end{align}
where $(g^{i\bar j})_{n\times n}=\overline{(g_{i\bar j})}_{n\times n}^{-1}$, and
$dV_{X}$ is the volume form with $ds_{X}^{2}$.
\end{lemma}

The symbols and notation are referred to \cite{guan-zhou-zhu10}.
See also \cite{cao-shaw-wang,siu96,siu00} and
\cite{Straube}.

\begin{lemma}\label{l:lem7}\quad
Let $(X,ds_{X}^{2})$ be a Hermitian manifold of dimension n with a
Hermitian metric $ds_{X}^{2}$, $\Omega\subset\subset X$ be a
strictly pseudoconvex domain in $X$ with $C^\infty$ boundary
$b\Omega$ and $\Phi\in C^\infty(\overline{\Omega})$. Let $\lambda$
be a $\bar\partial$ closed smooth form of bi-degree $(n,1)$ on
$\overline{\Omega}$. Assume the inequality
$$|(\lambda,\alpha)_{\Omega,\Phi}|^{2}\leq C\int_{ \Omega}|\bar{\partial}^{*}_
{\Phi}\alpha|^{2}\frac{{\rm e}^{-\Phi}}{\mu}dV_{X}<\infty,$$
holds for all $(n,1)$-form $\alpha\in {\rm Dom}_{\Omega}
(\bar\partial^*)\cap {\rm Ker}(\bar\partial)\cap C^\infty_{(n,1)}(\overline \Omega)$,
where $\frac{1}{\mu}$ is an integrable positive function on $\Omega$ and
$C$ is a constant. Then there is a solution $u$ to the
equation $\bar{\partial}u=\lambda$ such that
$$\int_{\Omega}|u|^{2}\mu {\rm e}^{-\Phi}dV_{X}\leq C.$$
\end{lemma}

The proof is similar to the proof of Lemma 2.4 in \cite{berndtsson}.

\begin{lemma}[See \cite{ohsawa3,demailly2010}]\label{l:vector}\quad
Let $(X,\omega)$ be a K\"{a}hler manifold of dimension $n$ with a
K\"{a}hler metric $\omega$, $E$ be a holomorphic hermitian vector
bundle over $X$, and $\Omega\subset\subset X$ be a domain with
$C^\infty$ boundary $b\Omega$. Let $\eta,g>0$ be smooth functions on
$X$. Then for every form $\alpha\in \mathcal
{D}(X,\Lambda^{n,q}T_{X}^{*}\otimes E)$, which is the space of
smooth differential $(n,q)$-forms with values in $E$ with compact
support, we have
\begin{equation}
\label{}
\|(\eta+g^{-1})^{\frac{1}{2}}D''^{*}\alpha\|^{2}
+\|\eta^{\frac{1}{2}}D''\alpha\|^{2}
\geq\langle\langle[\eta\sqrt{-1}\Theta_{E}-\sqrt{-1}\partial\bar\partial\eta-
\sqrt{-1}g\partial\eta\wedge\bar\partial\eta,\Lambda_{\omega}]\alpha,\alpha\rangle\rangle.
\end{equation}
\end{lemma}

\begin{lemma}[See \cite{demailly99,demailly2010}]\label{l:vector7}\quad
Let $X$ be a complete K\"{a}hler manifold equipped with a $($not
necessarily complete$)$ K\"{a}hler metric $\omega$, and let $E$ be a
holomorphic hermitian vector bundle over $X$. Assume that there are
smooth and bounded functions $\eta$, $g>0$ on $X$ such that the
$($Hermitian$)$ curvature operator
$$B:=[\eta\sqrt{-1}\Theta_{E}-\sqrt{-1}\partial\bar\partial\eta-\sqrt{-1}g
\partial\eta\wedge\bar\partial\eta,\Lambda_{\omega}]$$
is positive definite everywhere on $\Lambda^{n,q}T^{*}_{X}\otimes E$, for some $q\geq 1$.
Then for every form $\lambda\in L^{2}(X,\Lambda^{n,q}T^{*}_{X}\otimes E)$ such that
$D''\lambda=0$ and $\int_{X}\langle B^{-1}\lambda,\lambda\rangle dV_{\omega}<\infty$,
there exists $u\in L^{2}(X,\Lambda^{n,q-1}T^{*}_{X}\otimes E)$ such that $D''u=\lambda$ and
$$\int_{X}(\eta+g^{-1})^{-1}|u|^{2}dV_{\omega}\leq\int_{X}\langle B^{-1}\lambda,
\lambda\rangle dV_{\omega}.$$
\end{lemma}\emph{}

\subsection{Properties of Stein manifolds}
In this subsection, we recall some well-known properties on Stein
manifolds.

\begin{lemma}[\!\!\cite{FN1980}]\label{FN1}\quad
Let $X$ be a Stein manifold and $\varphi \in PSH(X)$.
Then there exists a sequence
$\{\varphi_{n}\}_{n=1,2,\ldots}$ of smooth strongly plurisubharmonic functions such that
$\varphi_{n} \downarrow \varphi$.
\end{lemma}

\begin{lemma}\label{l:lem.retra}\quad
Let $X$ be a Stein manifold and $H$ be a closed complex submanifold of $X$. Then there
is an open neighborhood
$U$ of $H$ in $X$ and a holomorphic retraction $r:U\longrightarrow H$.
\end{lemma}

\subsection{Properties of polar functions}
In this subsection, we collect some lemmas on properties of polar
functions.

\begin{lemma}\label{l:lem9}\quad
Let $M$ be a complex manifold of dimension $n$, and $S$ be an
$(n-1)$-dimensional submanifold. Let $\Psi\in\Delta(S)$. There
exists a local coordinate $(z_{1},\ldots,z_{n})$ on a neighborhood
$U$ of $x$ such that $z_{n}=0$ on $S\cap U$ and such that
$\psi:=\Psi-\log|z_{n}|^{2}$ is continuous on $U$. Then we have
$d\lambda_{z}[\Psi]={\rm e}^{-\psi}d\lambda_{z'}$, where $d\lambda_{z}$
and $d\lambda_{z'}$ denote the Lebesgue measures on $U$ and $S\cap
U$.
\end{lemma}

\begin{proof}\quad
Note that $d\lambda_{z}[\log|z_{n}|^{2}] = d\lambda_{z'}$ for $z=(z',z_{n})$.
By the definition of generalized residue volume form $d\lambda_{z}[\Psi]$ and
the continuity of $\psi$, we can prove the lemma.
\end{proof}

\begin{lemma}\label{l:lem8}\quad
Let $M$ be a complex manifold of dimension $n$. If $\Delta(S)$ is non-empty, then
$G(z,S)\in \Delta(S)$.
\end{lemma}

\begin{proof}\quad
The proof is similar to the proof of Proposition 9 in \cite{ohsawa5}.
\end{proof}

\section{Proofs of main theorems}
In this section, we give proofs of our main results stated
above.

\subsection{Proof of Theorem \ref{t:guan-zhou3}}
By Remark \ref{r:extend}, it suffices to prove the theorem for the
case when $M$ is a Stein manifold and $S$ is a closed complex
submanifold.

Since $M$ is Stein, one can find a sequence of strictly pseudoconvex
domains $\{D_v\}_{v=1}^\infty$ with smooth boundaries satisfying
$D_v\subset\subset D_{v+1}$ for all $v$ and
$\bigcup^{\infty}_{v=1}D_v=M$.
For $\Psi<0$, by Lemma \ref{l:lem.retra} and convolutions, we can
choose a sequence of smooth functions
$\{\varphi_{v,m}\}_{m=1,2,\ldots}$ and
$\{\Psi_{v,m}\}_{m=1,2,\ldots}$ on $D_{v+1}$, which satisfy
$\varphi_{v,m}+\Psi_{v,m}$ (resp.
$\varphi_{v,m}+(1+\delta)\Psi_{v,m}$) decreasing convergent to
$\varphi+\Psi$ (resp. $\varphi+(1+\delta)\Psi$) on $D_{v}$ and
$\Psi_{v,m}|_{\overline{D_{v}}}<0$, such that
$\Psi_{v,m}+\varphi_{v,m}$ and $(1+\delta)\Psi_{v,m}+\varphi_{v,m}$
are both plurisubharmonic functions. Denote
$\Psi_{v}:=\Psi|_{D_v}$.

Since $M$ is Stein, there is a holomorphic section $\tilde{F}$ of
$K_{M}$ on $M$ such that $\tilde{F}|_{S}={f}$. Let $ds_{M}^{2}$ be a
K\"{a}hler metric on $M$, $dV_{M}$ is the volume form with respect
to $ds_{M}^{2}$.

Let $\{v_{t_0,\varepsilon}\}_{t_{0}\in\mathbb{R},\varepsilon\in(0,\frac{1}{4})}$ be
a family of smooth increasing convex functions on $\mathbb{R}$, which are continuous
functions on $\mathbb{R}\cup+\infty$, such that:

(1) $v_{t_{0},\varepsilon}(t)=t$ for $t\geq-t_{0}-\varepsilon$,
 $v_{t_{0},\varepsilon}(t)={\rm constant}$ for $t<-t_{0}-1+\varepsilon$;

(2) $v''_{t_0,\varepsilon}(t)$ are pointwise convergent to
 ${1}_{\{-t_{0}-1< t<-t_{0}\}}$, when $\varepsilon\to 0$, and
 $0\leq v''_{t_0,\varepsilon}(t)\leq 2$ for any $t\in \mathbb{R}$;

(3) $v'_{t_0,\varepsilon}(t)$ are pointwise convergent to
 $b_{t_{0}}(t)=\int_{-\infty}^{t}{1}_{\{-t_{0}-1< s<-t_{0}\}}ds$
 ($b_{t_{0}}$ is also a continuous function on $\mathbb{R}\cup+\infty$), when
 $\varepsilon\to 0$, and $0\leq v'_{t_0,\varepsilon}(t)\leq1$ for any $t\in \mathbb{R}$.

We can construct the family $\{v_{t_0,\varepsilon}\}_{t_{0}\in\mathbb{R},\varepsilon
\in(0,\frac{1}{4})}$ by the setting
\begin{align}
\label{}
v_{t_0,\varepsilon}(t)&:=\int_{-\infty}^{t}\int_{-\infty}^{t_{1}}\frac{1}
{1-2\varepsilon}{1}_{\{-t_{0}-1+\varepsilon< s<-t_{0}-\varepsilon\}}*
\rho_{\frac{1}{4}\varepsilon}dsdt_{1}\nonumber\\
&\quad-\int_{-\infty}^{0}\int_{-\infty}^{t_{1}}
\frac{1}{1-2\varepsilon}{1}_{\{-t_{0}-1+\varepsilon< s<-t_{0}-\varepsilon\}}*
\rho_{\frac{1}{4}\varepsilon}dsdt_{1},
\end{align}
where $\rho_{\frac{1}{4}\varepsilon}$ is the kernel of convolution satisfying
${\rm Supp}(\rho_{\frac{1}{4}\varepsilon})\subset (-\frac{1}{4}\varepsilon,\frac{1}{4}
\varepsilon)$.

Therefore we have
$$v''_{t_0,\varepsilon}(t)=\frac{1}{1-2\varepsilon}{1}_
{\{-t_{0}-1+\varepsilon< t<-t_{0}-\varepsilon\}}*\rho_{\frac{1}{4}\varepsilon},$$ and
$$v'_{t_0,\varepsilon}(t)=\int_{-\infty}^{t}\frac{1}{1-2\varepsilon}
{1}_{\{-t_{0}-1+\varepsilon< s<-t_{0}-\varepsilon\}}*\rho_
{\frac{1}{4}\varepsilon}ds.$$

Let $\eta=s(-v_{t_{0},m}\circ\Psi_{v,m})$ and $\phi=u(-v_{t_{0},\varepsilon}
\circ\Psi_{v,m})$,
where $s\in C^{\infty}((0,+\infty))$ satisfies $s\geq\frac{1}{\delta}$, and
$u\in C^{\infty}((0,+\infty))$ satisfies $\lim_{t\to+\infty}u(t)=
-\log(1+\frac{1}{\delta})$, such that $u''s-s''>0$, and $s'-u's=1$.
Let $\Phi=\varphi_{v,m}+\Psi_{v,m}+\phi$.

Now let
$\alpha={\sum}'_{|I|=n}\sum^n_{j=1}\alpha_{I\bar
j}dz^I \wedge d\bar z^{j}\in {\rm Dom}_{D_v} (\bar\partial^*)\cap
{\rm Ker}(\bar\partial)\cap C^\infty_{(n,1)}(\overline {D_v})$. By Lemma
\ref{l:lem3} and Cauchy-Schwarz inequality, for
$s\geq\frac{1}{\delta}$ and $\Psi_{v,m}+\varphi_{v,m}$ being a
plurisubharmonic function on $D_{v+1}$, we get
\begin{align}
\label{equ:4.1}
\int_{D_v}(\eta+g^{-1})|\bar{\partial}^{*}_{\Phi}\alpha|^{2}{\rm e}^{-\Phi}dV_{M}
&\geq\underset{|I|=n}{{\sum}'}\sum_{j,k=1}^{n}\int_{D_v}
(-\partial_{j}\bar{\partial}_{k}\eta+\eta\partial_{j}\bar{\partial}_{k}
\Phi-g(\partial_{j} \eta)\bar\partial_{k} \eta )\alpha_{I}^{j }\overline
{{\alpha}^{\bar Ik}}{\rm e}^{-\Phi}dV_{M}\nonumber\\
&\geq\underset{|I|=n}{{\sum}'}\sum_{j,k=1}^{n}\int_{D_v}
(-\partial_{j}\bar{\partial}_{k}\eta+\eta\partial_{j}\bar{\partial}_{k}\phi+\frac{1}
{\delta}\partial_{j}\bar{\partial}_{k}
(\Psi_{v,m}+\varphi_{v,m})\nonumber\\
&\quad-g(\partial_{j} \eta)\bar\partial_{k} \eta )
\alpha_{I}^{j }\overline{{\alpha}^{\bar Ik}}{\rm e}^{-\Phi}dV_{M},
\end{align}
where $g$ is a positive continuous function on $D_{v}$. We need some
calculations in order to determine $g$.

We have
\begin{align}
\label{}
\partial_{j}\bar{\partial}_{k}\eta&=-s'(-v_{t_0,\varepsilon}\circ \Psi_{v,m})
\partial_{j}\bar{\partial}_{k}(v_{t_0,\varepsilon}\circ \Psi_{v,m})\nonumber\\
&\quad+s''(-v_{t_0,\varepsilon}\circ \Psi_{v,m})\partial_{j}(v_{t_0,\varepsilon}
\circ \Psi_{v,m})
\bar{\partial}_{k}(v_{t_0,\varepsilon}\circ \Psi_{v,m}),
\end{align}
and
\begin{align}
\label{}
\partial_{j}\bar{\partial}_{k}\phi&=-u'(-v_{t_0,\varepsilon}\circ \Psi_{v,m})
\partial_{j}\bar{\partial}_{k}v_{t_0,\varepsilon}\circ \Psi_{v,m}\nonumber\\
&\quad+u''(-v_{t_0,\varepsilon}\circ \Psi_{v,m})\partial_{j}(v_{t_0,\varepsilon}\circ
\Psi_{v,m})\bar{\partial}_{k}(v_{t_0,\varepsilon}\circ \Psi_{v,m})
\end{align}
for any $j,k$ which satisfies $1\leq j,k\leq n$.

We have
\begin{align}
\label{}
&\sum_{1\leq j,k\leq n}(-\partial_{j}\bar{\partial}_{k}\eta+\eta\partial_{j}\bar
{\partial}_{k}\phi-g(\partial_{j} \eta)
\bar\partial_{k} \eta )\alpha_{I}^{j}\overline{{\alpha}^{\bar Ik}}\nonumber\\
&\quad=(s'-su')\sum_{1\leq j,k\leq n}\partial_{j}\bar{\partial}_{k}(v_{t_0,\varepsilon}
\circ \Psi_{v,m})\alpha_{I}^{j}\overline{{\alpha}^{\bar Ik}}\nonumber\\
&\qquad+((u''s-s'')-gs'^{2})\sum_{1\leq j,k\leq n}\partial_{j}
(-v_{t_0,\varepsilon}\circ \Psi_{v,m})\bar{\partial}_{k}(-v_{t_0,\varepsilon}
\circ \Psi_{v,m})\alpha_{I}^{j}\overline{{\alpha}^{\bar Ik}}\nonumber\\
&\quad=(s'-su')\sum_{1\leq j,k\leq n}((v'_{t_0,\varepsilon}\circ\Psi_{v,m})
\partial_{j}\bar{\partial}_{k}\Psi_{v,m}+(v''_{t_0,\varepsilon}\circ \Psi_{v,m})
\partial_{j}(\Psi_{v,m})\bar{\partial}_{k}(\Psi_{v,m}))\alpha_{I}^{j}
\overline{{\alpha}^{\bar Ik}}\nonumber\\
&\qquad+((u''s-s'')-gs'^{2})\sum_{1\leq j,k\leq n}\partial_{j}
(-v_{t_0,\varepsilon}\circ \Psi_{v,m})\bar{\partial}_{k}(-v_{t_0,\varepsilon}
\circ \Psi_{v,m})\alpha_{I}^{j}\overline{{\alpha}^{\bar Ik}}.
\end{align}
For simplicity, we omit composite item $(-v_{t_0,\varepsilon}\circ
\Psi_{v,m})$ after $s'-su'$ and $(u''s-s'')-gs'^{2}$ in the above
equalities.

Denote $g=\frac{u''s-s''}{s'^{2}}\circ(-v_{t_0,\varepsilon}\circ
\Psi_{v,m})$. We have
$\eta+g^{-1}=(s+\frac{s'^{2}}{u''s-s''})\circ(-v_{t_0,\varepsilon}
\circ \Psi_{v,m})$.

Since $\varphi_{v,m}+\Psi_{v,m}$ and $\varphi_{v,m}+(1+\delta)\Psi_{v,m}$ are both
plurisubharmonic, and $0\leq v'_{t_0,\varepsilon}\circ\Psi_{v,m}\leq 1$,
we have
\begin{align}
\label{equ:4.2}
&\sum_{1\leq j,k\leq n}[(1-v'_{t_0,\varepsilon}\circ\Psi_{v,m})\partial_{j}
\bar{\partial}_{k}(\Psi_{v,m}+\varphi_{v,m})\nonumber\\
&\quad+(v'_{t_0,\varepsilon}\circ\Psi_{v,m})\partial_{j}\bar{\partial}_{k}
(\varphi_{v,m}+(1+\delta)\Psi_{v,m})]\alpha_{I}^{j}\overline{{\alpha}^{\bar Ik}}\geq 0,
\end{align}
which means
\begin{equation}
\label{equ:4.0}
\sum_{1\leq j,k\leq n}\bigg(\frac{1}{\delta}\partial_{j}\bar{\partial}_{k}(\Psi_{v,m}+
\varphi_{v,m})
+(v'_{t_0,\varepsilon}\circ\Psi_{v,m})\partial_{j}\bar{\partial}_{k}\Psi_{v,m}\bigg)
\alpha_{I}^{j}\overline{{\alpha}^{\bar Ik}}\geq 0.
\end{equation}
Because of $v'_{t_0,\varepsilon}\geq 0$  and $s'-su'=1$, by inequalities
(\ref{equ:4.1}) and (\ref{equ:4.0}), we have
\begin{equation}
\label{equ:3.1}
\begin{split}
\int_{D_v}(\eta+g^{-1})|\bar{\partial}^{*}_{\Phi}\alpha|^{2}{\rm e}^{-\Phi}dV_{M}
\geq\int_{D_v}(v''_{t_0,\varepsilon}\circ{\Psi_{v,m}})
|\alpha\llcorner(\bar \partial \Psi_{v,m})^\sharp|^2{\rm e}^{-\Phi}dV_M.
\end{split}
\end{equation}

Let
$\lambda=\bar{\partial}[(1-v'_{t_0,\varepsilon}(\Psi_{v,m})){\tilde{F}}]$.
By the definition of contraction, Cauchy-Schwarz inequality and
inequality (\ref{equ:3.1}), we can see that
\begin{align}
|(\lambda,\alpha)_{D_v,\Phi}|^{2}&=|((v''_{t_0,\varepsilon}\circ{\Psi_{v,m}})
\bar\partial\Psi_{v,m}\wedge\tilde{F},\alpha)_{D_v,\Phi}|^{2}\nonumber
\\&=|((v''_{t_0,\varepsilon}\circ{\Psi_{v,m}})\tilde{F},\alpha\llcorner
(\bar\partial\Psi_{v,m})^\sharp)_{D_v,\Phi}|^{2}\nonumber\\
&\leq\int_{D_v}
(v''_{t_0,\varepsilon}\circ{\Psi_{v,m}})| \tilde{F}|^2{\rm e}^{-\Phi}dV_{M}
\int_{D_v}(v''_{t_0,\varepsilon}\circ{\Psi_{v,m}})
|\alpha\llcorner(\bar\partial\Psi_{v,m})^\sharp|^2{\rm e}^{-\Phi}dV_{M}\nonumber\\
&\leq
\int_{D_v}
(v''_{t_0,\varepsilon}\circ{\Psi_{v,m}})| \tilde{F}|^2{\rm e}^{-\Phi}dV_M
\int_{ D_v}(\eta+g^{-1})|\bar{\partial}^{*}_{\Phi}\alpha|^{2}{\rm e}^{-\Phi}dV_{M}.
\end{align}

By Lemma \ref{l:lem7}, we have $(n,0)$-form
$u_{v,t_0,m,\varepsilon}$ on $D_{v}$
 satisfying $\bar{\partial}u_{v,t_0,m,\varepsilon}=\lambda$, such that

\begin{equation}
 \label{equ:3.2}
 \int_{ D_v}|u_{v,t_0,m,\varepsilon}|^{2}(\eta+g^{-1})^{-1} {\rm e}^{-\Phi}dV_{M}
  \leq\int_{D_v}(v''_{t_0,\varepsilon}\circ{\Psi_{v,m}})| \tilde{F}|^2{\rm e}^{-\Phi}dV_M.
\end{equation}

Denote $\mu_{1}={\rm e}^{v_{t_0,\varepsilon}\circ\Psi_{v,m}}$,
$\mu=\mu_{1}{\rm e}^{\phi}$. Assume that we can choose $\eta$ and $\phi$
such that $\mu\leq C(\eta+g^{-1})^{-1}$, where $C$
is just the constant in Theorem \ref{t:guan-zhou3}.

Note that $v_{t_0,\varepsilon}(\Psi_{v,m})\geq\Psi_{v,m}$, then we obtain
\begin{equation}
\label{equ:3.8}
\begin{split}
\int_{ D_v}|u_{v,t_0,m,\varepsilon}|^{2} {\rm e}^{-\varphi_{v,m}}dV_{M}\leq
\int_{ D_v}|u_{v,t_0,m,\varepsilon}|^{2}\mu_{1}{\rm e}^{\phi}
{\rm e}^{-\varphi_{v,m}-\Psi_{v,m}-\phi}dV_{M}.
\end{split}
\end{equation}

By inequalities (\ref{equ:3.2}) and (\ref{equ:3.8}), we can obtain that
$$\int_{ D_v}|u_{v,t_0,m,\varepsilon}|^{2} {\rm e}^{-\varphi_{v,m}}dV_{M}\leq C\int_{D_v}
(v''_{t_0,\varepsilon}\circ{\Psi_{v,m}})|
\tilde{F}|^2{\rm e}^{-\Phi}dV_M,$$ under the assumption $\mu\leq C
(\eta+g^{-1})^{-1}$. As $\varphi$ is continuous, $\phi$ and
$(v''_{t_0,\varepsilon}\circ{\Psi_{v,m}})|
\tilde{F}|^2{\rm e}^{-\Psi_{v,m}}$ are uniformly bounded on
$\overline{D}_{v}$ independent of $m$ (because of
${\rm Supp}(v_{t_{0},\varepsilon})\subset\subset(-t_{0}-1,-t_{0})$), it is
clear that $\int_{D_v} (v''_{t_0,\varepsilon}\circ{\Psi_{v,m}})|
\tilde{F}|^2{\rm e}^{-\Phi}dV_M$ are uniformly bounded independent of $m$,
for any given $v$, $t_0$ and $\varepsilon$. By weakly compactness of
the unit ball of $L^{2}_{\varphi}(D_{v})$ and dominated convergence
theorem, we can see that the weak limit of some weakly convergent
subsequence of $\{u_{v,t_{0},m,\varepsilon}\}_{m}$ when
$m\to+\infty$ gives an $(n,0)$-form $u_{v,t_0,\varepsilon}$ on $D_v$
satisfying
\begin{equation}
\label{equ:3.3}\int_{ D_v}|u_{v,t_0,\varepsilon}|^{2} {\rm e}^{-\varphi}dV_{M}
\leq\frac{{C}}{{\rm e}^{A_{t_0}}}\int_{D_v}
(v''_{t_0,\varepsilon}\circ{\Psi_{v}})| \tilde{F}|^2{\rm e}^{-\varphi-\Psi_{v}}dV_M,
\end{equation}
where $A_{t_0}:=\sup_{t\geq t_0}\{u(t)\}$. As $\lim_{t\to+\infty}u(t)=
-\log(1+\frac{1}{\delta})$, it is clear that $\lim_{t_{0}\to\infty}\frac{1}
{{\rm e}^{A_{t_0}}}=1+\frac{1}{\delta}$.

As $\varphi_{v,m}+\Psi_{v,m}$ is decreasing to $\varphi+\Psi$ on
$M$, then for any given $t_{0}$ there exists $m_{0}$ and a
neighborhood $U_{0}$ of $\{\Psi=-\infty\}\cap \overline{D_{v}}$ in
$M$, such that for any $m\geq m_{0}$ and $\varepsilon$,
$v''_{t_0,\varepsilon}\circ\Psi_{v,m}|_{U_{0}}=0$. Now we will show
that $u_{v,t_0,\varepsilon}$ satisfies
$u_{v,t_0,\varepsilon}|_{S}=0$.

As $v''_{t_0,\varepsilon}\circ\Psi_{v,m}|_{U_{0}}=0$ for $m>m_{0}$,
it is clear that $\bar\partial u_{v,t_0,m,\varepsilon}|_{U_0}=0$,
which implies that $u_{v,t_0,\varepsilon}|_{U_0}$ is holomorphic.

Note that $\int_{ D_v}|u_{v,t_{0},m,\varepsilon}|^{2}
{\rm e}^{-\varphi_{v,m}}dV_{M}$ have uniform bound independent of $m$.
Then we can choose a subsequence with respect to $m$ from the chosen
weakly convergent subsequence of $u_{v,t_{0},m,\varepsilon}$, such
that the subsequence is uniformly convergent on any compact subset
of $U_0$, and we still denote the subsequence by
$u_{v,t_{0},m,\varepsilon}$ without ambiguity.

By inequality (\ref{equ:3.2}), we can see that $\int_{
D_v}|u_{v,t_0,m,\varepsilon}|^{2} (\eta+g^{-1})^{-1}
{\rm e}^{-\varphi_{v,m}-\phi-\Psi_{v,m}}dV_{M}$ are uniformly bounded
independent of $m$.

Note
$(\eta+g^{-1})^{-1}=(s(-v_{t_{0},\varepsilon}\circ\Psi_{v,m})+
\frac{s'^{2}}{u''s-s''}\circ(-v_{t_0,\varepsilon}\circ
\Psi_{v,m}))^{-1}$
and
${\rm e}^{-\varphi_{v,m}-\phi}={\rm e}^{-\varphi_{v,m}-u(-v_{t_{0},\varepsilon}\circ\Psi_{v,m})}$
have positive uniform bound independent of $m$. Then $\int_{
K_0}|u_{v,t_0,m,\varepsilon}|^{2}{\rm e}^{-\Psi_{v,m}}dV_{M}$ have uniform
bound independent of $m$ for any compact set $K_{0}\subset\subset
U_0\cap D_v$.

As $\Psi_{v,m'}+\varphi_{v,m'}\leq \Psi_{v,m}+\varphi_{v,m}$, where
$m'\geq m$, we have
$$|u_{v,t_0,m',\varepsilon}|^{2}{\rm e}^{-\Psi_{v,m}-\varphi_{v,m}}\leq|u_{v,t_0,m',
\varepsilon}|^{2}{\rm e}^{-\Psi_{v,m'}-\varphi_{v,m'}}.$$

Note that $\varphi_{v,m}$ have uniform bound independent of $m$,
then $\int_{K_0}|u_{v,t_0,m',\varepsilon}|^{2}{\rm e}^{-\Psi_{v,m}}dV_{M}$
have uniform bound independent of $m$ and $m'$, for any compact set
$K_{0}\subset\subset U_0\cap D_v$.

It is clear that $\int_{
K_0}|u_{v,t_0,\varepsilon}|^{2}{\rm e}^{-\Psi_{v,m}}dV_{M}$ have uniformly
bound independent of $m$, for any compact set $K_{0}\subset\subset
U_0\cap D_v$. Then we have $\int_{
K_0}|u_{v,t_0,\varepsilon}|^{2}{\rm e}^{-\Psi_{v}}dV_{M}<\infty$, for any
compact set $K_{0}\subset\subset U_0\cap D_v$.

Because of $\Psi\in\Delta_{\varphi,\delta}(S)$ and
$\Psi_{v}=\Psi|_{D_v}$, and by condition (2) in the definition of
$\#(S)$, we see that ${\rm e}^{-\Psi}$ is disintegrable near $S$. Then it
is clear that $u_{v,t_0,\varepsilon}$ satisfies
$u_{v,t_0,\varepsilon}|_{S}=0$.

Let $F_{v,t_0,\varepsilon}:=(1-v'_{t_0,\varepsilon}\circ\Psi_{v})\widetilde{F}
-u_{v,t_0,\varepsilon}$. By inequality (\ref{equ:3.3}) and $u_{v,t_0,\varepsilon}|_{S}=0$,
we can see that
$F_{v,t_0,\varepsilon}$ is a holomorphic $(n,0)$-form on $D_{v}$ satisfying
$F_{v,t_0,\varepsilon}|_{S}=\tilde{F}|_{S}$ and

\begin{equation}
\label{equ:3.5}
\int_{ D_v}|F_{v,t_0,\varepsilon}-(1-v'_{t_0,\varepsilon}\circ\Psi_{v})
\tilde{F}|^{2}{\rm e}^{-\varphi}dV_{M}
\leq\frac{C}{{\rm e}^{A_{t_0}}}\int_{D_v}(v''_{t_0,\varepsilon}\circ\Psi_{v})
|\tilde{F}|^{2}{\rm e}^{-\varphi-\Psi_{v}}dV_{M}.
\end{equation}

Given $t_0$ and $D_{v}$, it is clear that
$(v''_{t_0,\varepsilon}\circ\Psi_{v})|
\tilde{F}|^{2}{\rm e}^{-\varphi-\Psi_{v}}$ have uniform bound on $D_{v}$
independent of $\varepsilon$. Then
$\int_{D_v}|(1-v'_{t_0,\varepsilon}\circ\Psi_{v})\tilde{F}|^{2}
{\rm e}^{-\varphi}dV_{M}$ and
$\int_{D_v}v''_{t_0,\varepsilon}\circ\Psi_{v}|\tilde{F}|^{2}{\rm e}^{-\varphi-\Psi_{v}}dV_{M}$
have uniform bound independent of $\varepsilon$, for any given $t_0$
and $D_v$.

As $\bar\partial F_{v,t_{0},\varepsilon}=0$ and weakly compactness
of the unit ball of $L^{2}_{\varphi}(D_{v})$, we see that the weak
limit of some weakly convergent subsequence of
$\{F_{v,t_0,\varepsilon}\}_{\varepsilon}$ when $\varepsilon\to 0$
gives us a holomorphic $(n,0)$-form $F_{v,t_0}$ on $D_{v}$
satisfying $F_{v,t_0}|_{S}=\tilde{F}|_{S}$.

Note that one can also choose a subsequence of the weakly convergent
subsequence of $\{F_{v,t_0,\varepsilon}\}_{\varepsilon}$, such that
the chosen sequence is uniformly convergent on any compact subset of
$D_v$, still denoted by $\{F_{v,t_0,\varepsilon}\}_{\varepsilon}$
without ambiguity.

For any compact subset $K$ in $D_v$, it is clear that
$F_{v,t_0,\varepsilon}$,
$(1-v'_{t_0,\varepsilon}\circ\Psi_{v})\tilde{F}$ and
$(v''_{t_0,\varepsilon}\circ\Psi_{v})|\tilde{F}|^{2}{\rm e}^{-\varphi-\Psi_{v}}$
have uniform bound on $K$ independent of $\varepsilon$.

By using dominated convergence theorem on any compact subset $K$ of
$D_v$ and inequality (\ref{equ:3.5}), we have
\begin{equation}
\int_{K}|F_{v,t_0}-(1-b_{t_0}(\Psi_{v}))\tilde{F}|^{2}{\rm e}^{-\varphi}dV_{M}
\leq\frac{C}{{\rm e}^{A_{t_0}}}\int_{D_v}({1}_{\{-t_{0}-1< t<-t_{0}\}}
\circ\Psi_{v})|\tilde{F}|^{2}{\rm e}^{-\varphi-\Psi_{v}}dV_{M},
\end{equation}
which implies
\begin{equation}
\label{equ:3.4}
\int_{ D_v}|F_{v,t_0}-(1-b_{t_0}(\Psi_{v}))\tilde{F}|^{2}{\rm e}^{-\varphi}dV_{M}
\leq\frac{C}{{\rm e}^{A_{t_0}}}\int_{D_v}({1}_{\{-t_{0}-1< t<-t_{0}\}}
\circ\Psi_{v})|\tilde{F}|^{2}{\rm e}^{-\varphi-\Psi_{v}}dV_{M}.
\end{equation}

By the definition of $dV_{M}[\Psi]$ and $\sum_{k=1}^{n}\frac{\pi^{k}}{k!}
\int_{S_{n-k}}|f|^{2}{\rm e}^{-\varphi}dV_{M}[\Psi]<\infty$, we have
\begin{align}
\label{equ:3.6}
&\limsup_{t_{0}\to+\infty}\int_{D_v}({1}_{\{-t_{0}-1< t<-t_{0}\}}
\circ\Psi_{v})|\tilde{F}|^{2}{\rm e}^{-\varphi-\Psi_{v}}dV_{M}\nonumber\\
&\quad\leq
\limsup_{t_{0}\to+\infty}\int_{M}{1}_{\overline{D}_{v}}
({1}_{\{-t_{0}-1<t<-t_{0}\}}\circ\Psi)|\tilde{F}|^{2}{\rm e}^{-\varphi-\Psi}dV_{M}\nonumber\\
&\quad\leq\sum_{k=1}^{n}\frac{\pi^{k}}{k!}\int_{S_{n-k}}{1}_
{\overline{D}_{v}}|f|^{2}{\rm e}^{-\varphi}dV_{M}[\Psi]
\leq\sum_{k=1}^{n}\frac{\pi^{k}}{k!}\int_{S_{n-k}}|f|^{2}{\rm e}^{-\varphi}dV_{M}[\Psi]<\infty.
\end{align}
Then $\int_{D_v}({1}_{\{-t_{0}-1< t<-t_{0}\}}\circ\Psi_{v})|
\tilde{F}|^{2}{\rm e}^{-\varphi-\Psi_{v}}dV_{M}$
have uniform bound independent of $t_{0}$, for any given $D_v$, and
\begin{align}
\label{equ:3.7}
&\limsup_{t_{0}\to+\infty}\int_{D_v}({1}_{\{-t_{0}-1< t<-t_{0}\}}
\circ\Psi_{v})|\tilde{F}|^{2}{\rm e}^{-\varphi-\Psi_{v}}dV_{M}\nonumber\\
&\quad\leq\sum_{k=1}^{n}\frac{\pi^{k}}{k!}\int_{S_{n-k}}|f|^{2}{\rm e}^{-\varphi}dV_{M}
[\Psi]<\infty.
\end{align}

It is clear that $\int_{
D_v}|F_{v,t_0}-(1-b_{t_0}(\Psi_{v}))\tilde{F}|^{2}
{\rm e}^{-\varphi}dV_{M}$ have uniform bound independent of $t_{0}$, for
any given $D_v$.

As $\int_{
D_v}|(1-b_{t_0}(\Psi_{v}))\tilde{F}|^{2}{\rm e}^{-\varphi}dV_{M}$ have
uniform bound independent of $t_{0}$, by inequality (\ref{equ:3.4})
and
\begin{align}
\label{equ:3.9}
&\bigg(\int_{ D_v}|F_{v,t_0}-(1-b_{t_0}(\Psi_{v}))\tilde{F}|^{2}{\rm e}^{-\varphi}dV_{M}\bigg)
^{\frac{1}{2}}+\bigg(\int_{ D_v}|(1-b_{t_0}(\Psi_{v}))\tilde{F}|^{2}{\rm e}^{-\varphi}dV_{M}\bigg)^
{\frac{1}{2}}\nonumber\\
&\quad\geq\bigg(\int_{ D_v}|F_{v,t_0}|^{2}{\rm e}^{-\varphi}dV_{M}\bigg)^{\frac{1}{2}},
\end{align}
we can obtain that $\int_{ D_v}|F_{v,t_0}|^{2}{\rm e}^{-\varphi}dV_{M}$
have uniform bound independent of $t_{0}$.

By $\bar\partial F_{v,t_{0}}=0$ and weakly compactness of the unit
ball of $L^{2}_{\varphi}(D_{v})$, one can see that the weak limit of
some weakly convergent subsequence of $\{F_{v,t_0}\}_{t_0}$ when
$t_{0}\to+\infty$ gives us a holomorphic $(n,0)$-form $F_{v}$ on
$D_{v}$ satisfying $F_{v}|_{S}=\tilde{F}|_{S}$.

Note that we can also choose a subsequence of the weakly convergent
subsequence of $\{F_{v,t_{0}}\}_{t_{0}}$, such that the chosen
sequence is uniformly convergent on any compact subset of $D_v$,
denoted again by $\{F_{v,t_0}\}_{t_{0}}$ without ambiguity.

For any compact subset $K$ on $D_v$, it is clear that both
$F_{v,t_0}$ and $(1-b_{t_0}\circ\Psi_{v})\tilde{F}$ have uniform
bound on $K$ independent of $t_0$.

By inequalities (\ref{equ:3.4}), (\ref{equ:3.7}) and dominated
convergence theorem on any compact subset $K$ of $D_v$, we have
\begin{equation}
\begin{split}
&\int_{D_v}{1}_{K}|F_{v}|^{2}{\rm e}^{-\varphi}dV_{M}
\leq\frac{C}{{\rm e}^{A_{t_0}}}\sum_{k=1}^{n}\frac{\pi^{k}}{k!}
\int_{S_{n-k}}|f|^{2}{\rm e}^{-\varphi}dV_{M}[\Psi],
\end{split}
\end{equation}
which implies
\begin{equation}
\begin{split}
&\int_{ D_v}|F_{v}|^{2}{\rm e}^{-\varphi}dV_{M}
\leq\frac{C}{{\rm e}^{A_{t_0}}}\sum_{k=1}^{n}\frac{\pi^{k}}{k!}\int_{S_{n-k}}
|f|^{2}{\rm e}^{-\varphi}dV_{M}[\Psi].
\end{split}
\end{equation}
Note that the Lebesgue measure of $\{\Psi=-\infty\}$ is zero.

Now it suffices to find $\eta$ and $\phi$ such that
$(\eta+g^{-1})\leq
C{\rm e}^{-\Psi_{v,\delta}}{\rm e}^{-\phi}=C\mu^{-1}$ on
$D_v$. As $\eta=s(-v_{t_0,\varepsilon}\circ\Psi_{v,\delta})$ and
$\phi=u(-v_{t_0,\varepsilon}\circ\Psi_{v,\delta})$, we have
$(\eta+g^{-1}) {\rm e}^{v_{t_0,\varepsilon}\circ\Psi_{v,\delta}}{\rm e}^{\phi}=
(s+\frac{s'^{2}}{u''s-s''}){\rm e}^{-t}{\rm e}^{u}\circ(-v_{t_0,\varepsilon}\circ\Psi_{v,\delta})$.

We are naturally led to obtain the following system of ODEs:
\begin{equation}
\label{GZ}
\begin{split}
&(1)\ \bigg(s+\frac{s'^{2}}{u''s-s''}\bigg){\rm e}^{u-t}={C}, \\
&(2)\ s'-su'=1,
\end{split}
\end{equation}
where $t\in[0,+\infty)$, and ${C}=1$.

One can solve the ODE (\ref{GZ}) (for details, see the following remark) to
get $u=-\log(1+\frac{1}{\delta}-{\rm e}^{-t})$ and
$s=\frac{(1+\frac{1}{\delta})t+\frac{1}{\delta}(1+\frac{1}{\delta})}
{1+\frac{1}{\delta}-{\rm e}^{-t}}-1$.

One may check that $s\in C^{\infty}((0,+\infty))$ satisfies
$s\geq\frac{1}{\delta}$,
$\lim_{t\to+\infty}u(t)=-\log(1+\frac{1}{\delta})$ and $u\in
C^{\infty}((0,+\infty))$ satisfies $u''s-s''>0$.

\begin{remark}\quad
Now we solve the equation (\ref{GZ}):

By (2) of equation (\ref{GZ}), we have  $su''-s''=-s'u'$. Then (1) of equation
(\ref{GZ}) can be changed into
$$\bigg(s-\frac{s'}{u'}\bigg){\rm e}^{u-t}=C,$$ which is
$$\frac{su'-s'}{u'}{\rm e}^{u-t}=C.$$
By (2) of equation (\ref{GZ}), we have
$$C=\frac{su'-s'}{u'}{\rm e}^{u-t}=\frac{-1}{u'}{\rm e}^{u-t},$$
which is
$$\frac{d{\rm e}^{-u}}{dt}=-u'{\rm e}^{-u}=\frac{{\rm e}^{-t}}{C}.$$ Note that (2) of equation
(\ref{GZ}) is equivalent to $\frac{d(s{\rm e}^{-u})}{dt}={\rm e}^{-u}$. As $s\geq
0$, we obtain the solution
\begin{displaymath}
     \begin{cases}
      u=-\log(a-{\rm e}^{-t}), \\
      \displaystyle s=\frac{at+{\rm e}^{-t}+b}{a-{\rm e}^{-t}},\end{cases}
    \end{displaymath}
when $C=1$, where $a\geq 1$ and $b\geq-1$.

As $\lim_{t\to+\infty}u(t)= -\log(1+\frac{1}{\delta})$, we have
$a=(1+\frac{1}{\delta})$.

By $su''-s''=-s'u'$, it is clear that $u''s-s''>0$ is equivalent to
$s'>0$, which implies $b\leq a^{2}-2a$ (by considering the limit at
$0$ and $\infty$ of $s'$).

As $s\geq\frac{1}{\delta}$, we have $b=
a^{2}-2a=\frac{1}{\delta^{2}}-1$.
\end{remark}

Define $F_v=0$ on $M\backslash D_v$. As $\lim_{t_{0}\to \infty}A_{t_{0}}=
-\log(1+\frac{1}{\delta})$, then the weak limit of some weakly convergent subsequence of
$\{F_v\}_{v=1}^\infty$ gives us a holomorphic $(n,0)$-form $F$ on $M$ satisfying
$F|_{S}=\tilde{F}|_{S}$, and
$$c_n\int_{M}{\rm e}^{-\varphi}F\wedge\bar F=\int_{ M}|F|^{2}{\rm e}^{-\varphi}dV_{M}
\leq C\bigg(1+\frac{1}{\delta}\bigg)\sum_{k=1}^{n}\frac{\pi^{k}}{k!}\int_{S_{n-k}}|f|^{2}
{\rm e}^{-\varphi}dV_{M}[\Psi],$$
where $c_{k}=(-1)^{\frac{k(k-1)}{2}}\sqrt{-1}^k$ for $k\in\mathbb{Z}$.

In conclusion, we have proved Theorem \ref{t:guanzhou1} with the
constant $C=1$.


\subsection{Proof of Remark \ref{r:guan-zhou3}}
Let $\Delta$ be the unit disc on $\mathbb{C}$, with coordinate $z$.
Let $$\varphi(z)=(1+\delta)\max\{\log|z|^{2},\log|a|^{2}\},\quad {\rm and}\quad
\Psi(z)=-\max\{\log|z|^{2},\log|a|^{2}\}+\log|z|^{2}-\varepsilon,$$
where $a\in(0,1)$, $\varepsilon>0$.

As $\varphi$ and $\varphi+(1+\delta)\Psi$ are both plurisubharmonic,
and
$$\varphi+\Psi=\frac{\delta\varphi+(\varphi+(1+\delta)\Psi)}{1+\delta},$$ it is
clear that $\Psi(z)\in \Delta_{\varphi,\delta}(S)$, where $S=\{z=0\}$.

For any $f(0)\neq 0$, it suffices to prove
\begin{equation}
\label{equ:opt1}
\lim_{a\to 0}\frac{\min_{F\in Hol(\Delta)}\int_{\Delta}
|F|^{2}{\rm e}^{-\varphi}d\lambda}{a^{-2\delta}{\rm e}^{\varepsilon}|F(0)|^{2} }=
(1+\delta^{-1})\frac{\pi}{{\rm e}^{\varepsilon}},
\end{equation}
where $F(0)=f(0)$.
It is because ${\rm e}^{-\varphi}d\lambda[\Psi]=a^{-2\delta}{\rm e}^{\varepsilon}\delta_{0}$
(by Lemma \ref{l:lem9}), $\delta_{0}$ is the dirac function at $0$, when
$\varepsilon$ goes to zero,
then we can see that the constant of Theorem \ref{t:guan-zhou3} is optimal.

By Taylor expansion, at $0\in\mathbb{C}$, $F(z)=\sum_{k=0}^{\infty}a_{k}z^{k}$,
where $a_{k}$ are complex constants. Note that $\int_{\Delta}z^{k_{1}}\bar{z}^
{k_{2}}{\rm e}^{-\varphi}d\lambda=0$ when $k_{1}\neq k_{2}$, and
$\int_{\Delta}z^{k_{1}}\bar{z}^{k_{2}}{\rm e}^{-\varphi}d\lambda>0$ when $k_{1}= k_{2}$. It
is clear that
$$\min_{F\in {\rm Hol}(\Delta)}\int_{\Delta}|F|^{2}{\rm e}^{-\varphi}d\lambda=
\int_{\Delta}|F(0)|^{2}{\rm e}^{-\varphi}d\lambda.$$
As
$$\int_{\Delta}{\rm e}^{-\varphi}d\lambda=\pi\bigg(\frac{a^{-2\delta}-1}{\delta}+a^{-2\delta}\bigg) \quad {\rm and}\quad
\lim_{a\to0}\frac{\frac{a^{-2\delta}-1}{\delta}+a^{-2\delta}}{a^{-2\delta}}=
1+\frac{1}{\delta},$$
we can prove the equality (\ref{equ:opt1}).

In the following part, we will show that for any given planar domain and point in the
domain, the constant $C$ is also optimal.

Let $\Omega$ be a planar domain in $\mathbb{C}$, such that unit disc $\Delta\subset\Omega$.
Let $$\varphi_{N}(z)=(1+\delta)\max\{\log|z|^{2},\log|a|^{2},N\log|z|^{2}\},$$ and
$$\Psi_{N}(z)=-\max\{\log|z|^{2},\log|a|^{2},N\log|z|^{2}\}+\log|z|^{2}-\varepsilon,$$
where $a\in(0,1)$, $\varepsilon>0$, $N>3$.
It is clear that $\varphi_{N}(z)|_{\Delta}=\varphi(z)$, and
$\Psi_{N}(z)|_{\Delta}=\Psi(z)$.

By the same arguments as above, we can see that
$$\min_{F\in {\rm Hol}(\Omega)}\int_{\Omega}|F|^{2}{\rm e}^{-\varphi_{N}}
d\lambda\geq\int_{\Delta}|F(0)|^{2}{\rm e}^{-\varphi_{N}}d\lambda.$$

By the similar calculations, we can obtain that for any
$f(0)\neq 0$,
\begin{equation}
\label{}
\lim_{a\to 0}\frac{\min_{F\in {\rm Hol}(\Omega)}\int_{\Omega} |F|^{2}{\rm e}^{-\varphi_{N}}
d\lambda}{a^{-2\delta}{\rm e}^{\varepsilon}|F(0)|^{2}}
\geq
\lim_{a\to 0}\frac{\min_{F\in {\rm Hol}(\Delta)}\int_{\Delta} |F|^{2}{\rm e}^{-\varphi}
d\lambda}{a^{-2\delta}{\rm e}^{\varepsilon}|F(0)|^{2}}
=(1+\delta^{-1})\frac{\pi}{{\rm e}^{\varepsilon}},
\end{equation}
where $F(0)=f(0)$.

Then we have shown that for any given planar domain and point in the
domain, the constant $C$ is also optimal.

\subsection{Proof of Theorem \ref{t:guanzhou1}}
By Remark \ref{r:extend}, it suffices to prove the case when $M$ is
a Stein manifold.

Since $M$ is Stein, we can find a sequence of strictly pseudoconvex
domains $\{D_v\}_{v=1}^\infty$ with smooth boundaries satisfying
$D_v\subset\subset D_{v+1}$ for all $v$ and
$\bigcup^{\infty}_{v=1}D_v=M$.

As $\Psi<0$, by Lemma \ref{FN1}, we can choose a sequence of smooth
plurisuharmonic functions $\{\Psi_{v,m}\}_{m=1,2,\ldots}$ on $M$,
such that $\Psi_{v,m}$ decreasingly converge to $\Psi$ on $D_{v}$
and $\Psi_{v,m}|_{\overline{D_{v}}}<0$. Denote
$\Psi_{v}:=\Psi|_{D_v}$.

Just like the arguments in \cite{ohsawa-takegoshi} and
\cite{ohsawa3}, we can assume that $\varphi$ is smooth on $M$ as in
Theorem \ref{t:guanzhou1}.

Since $M$ is Stein, there is a holomorphic section $\tilde{F}$ of
$K_{M}$ on $M$ such that $\tilde{F}|_{S}={f}$.

Let $ds_{M}^{2}$ be a K\"{a}hler metric on $M$, and $dV_{M}$ be the
volume form with respect to $ds_{M}^{2}$. As $\Psi$ is
plurisubharmonic, it is clear that for any $\delta>0$, we have
$\Psi\in\Delta_{\varphi,\delta}(S)$.

When $\delta$ approaches to infinity, we obtain Theorem \ref{t:guanzhou1} by
Theorem \ref{t:guan-zhou3}.


\subsection{Proof of Corollary \ref{c:ohsawa}}
As $\kappa_{M/S}$ has extremal property, we have
\begin{equation}
\kappa_{M/S}=\sup_{f}
\frac{f\otimes\bar{f}}{2^{-n}\sum_{k=1}^{n}\int_{S_{n-k}}
\frac{c_{n}f\wedge \bar f}{dV_{M}} dV_{M}[G(\cdot,S)]}
=\sup_{f}
 \frac{f\otimes\bar{f}}{2^{-n}\int_{S_{n-k}} \frac{c_{n}f\wedge \bar f}{dV_{M}}
 dV_{M}[G(\cdot,S)]},
\end{equation}
for all holomorphic section $f$ of $K_{X}|_{S_{n-k}}$.

By Lemma \ref{l:lem8}, we have $G(\cdot,S)\in \Delta(S)$. Denote
$\Psi:=G(\cdot,S)$.

By Theorem \ref{t:guanzhou1}, we have
\begin{align}
\label{}
\frac{F\otimes\bar{F}}{2^{-n}c_{n}\int_{M} F\wedge\bar{F}}&\geq
\frac{f\otimes\bar{f}}
{2^{-n}\sum_{k=1}^{n}\frac{\pi^{k}}{k!}\int_{S_{n-k}}
\frac{c_{n}f\wedge \bar f}{dV_{M}} dV_{M}[G(\cdot,S)]}\nonumber\\
&=\frac{f\otimes\bar{f}}{2^{-n}\frac{\pi^{k}}{k!}\int_{S_{n-k}} \frac{c_{n}f\wedge
\bar f}{dV_{M}} dV_{M}[G(\cdot,S)]}.
\end{align}
Then we can obtain that
$$\frac{\pi^{k}}{k!}\frac{F\otimes\bar{F}}{2^{-n}c_{n}\int_{M} F\wedge\bar{F}}\geq
 \frac{f\otimes\bar{f}}{2^{-n}\int_{S_{n-k}} \frac{c_{n}f\wedge \bar f}{dV_{M}}
 dV_{M}[G(\cdot,S)]},$$
where $F$ is an extension of $f$.

As
$$\kappa_{M/S}=\sup_{f} \frac{f\otimes \bar f}{(2^{-n}\int_{S_{n-k}} \frac{c_{n}
f\wedge \bar f}{dV_{M}} dV_{M}[(G\cdot,S)])},$$ we prove the
corollary.

\subsection{Proof of Theorem \ref{t:guan-zhou}}
Just like the arguments in \cite{ohsawa-takegoshi} and
\cite{ohsawa3}, one may assume that $H=S$ and that both $\varphi$
and $\psi$ are smooth on $X$ as in Theorem \ref{t:guan-zhou}. Let
$\Psi=\log|w^{2}|+\psi$. By Lemma \ref{l:lem9}, we have
\begin{equation}
\label{}
d\lambda_{z}[\Psi]=d\lambda_{z}[\log|w^{2}|+\psi]={\rm e}^{-\psi}d\lambda_{z'}.
\end{equation}

By Theorem \ref{t:guanzhou1}, we get the theorem.

\subsection{Proof of Theorem \ref{t:guan-zhou-vector}}
By Remark \ref{r:extend}, it suffices to prove the theorem for the
case when $M$ is a Stein manifold.

Since $M$ is a Stein manifold, we can find a sequence of Stein
manifolds $\{D_v\}_{v=1}^\infty$ satisfying $D_v\subset\subset
D_{v+1}$ for all $v$ and
$\bigcup^{\infty}_{v=1}D_v=M$, and $D_{v}\setminus
S$ are all complete K\"{a}hler. Let $\Psi_{v}:=\Psi|_{D_{v}}$.

Since $M$ is Stein, there is a holomorphic section $\tilde{F}$ of
$K_{M}$ on $M$ such that $\tilde{F}|_{S}={f}$. Let $ds_{M}^{2}$ be a
K\"{a}hler metric on $M$, and $dV_{M}$ is the volume form with
respect to $ds_{M}^{2}$.

Let
$\{v_{t_0,\varepsilon}\}_{t_{0}\in\mathbb{R},\varepsilon\in(0,\frac{1}{4})}$
be family of smooth increasing convex functions on $\mathbb{R}$
(also continuous functions on $\mathbb{R}\cup-\infty$) which are the
same as in the proof of Theorem \ref{t:guan-zhou3}.

Therefore we have
$$v''_{t_0,\varepsilon}=\frac{1}{1-2\varepsilon}{1}_
{\{-t_{0}-1+\varepsilon<
t<-t_{0}-\varepsilon\}}*\rho_{\frac{1}{4}\varepsilon},\quad {\rm and}\quad
v'_{t_0,\varepsilon}=\int_{-\infty}^{t}\frac{1}{1-2\varepsilon}{1}_
{\{-t_{0}-1+\varepsilon<
s<-t_{0}-\varepsilon\}}*\rho_{\frac{1}{4}\varepsilon}ds.$$

Let $\eta=s(-v_{t_{0},m}\circ\Psi)$ and
$\phi=u(-v_{t_{0},\varepsilon}\circ\Psi)$, where $s\in
C^{\infty}((0,+\infty))$ satisfies $s\geq \frac{1}{\delta}$, and
$u\in C^{\infty}((0,+\infty))$ satisfies $\lim_{t\to+\infty}u(t)=
-\log(1+\frac{1}{\delta})$, such that $u''s-s''>0$, and $s'-u's=1$.
Let $h'=h{\rm e}^{-\Psi-\phi}$.

Now let $\alpha\in \mathcal{D}(X,\Lambda^{n,1}T_{M\setminus
S}^{*}\otimes E)$ with compact support on $M\setminus S$. By Lemma
\ref{l:vector}, $s\geq \frac{1}{\delta}$ and
$\Theta_{h{\rm e}^{-\Psi}}\geq 0$ on $M\setminus S$, we get
\begin{align}
\label{equ:10.1}
&\|(\eta+g^{-1})^{\frac{1}{2}}D''^{*}\alpha\|^{2}_{D_v\setminus
S,h'} +\|\eta^{\frac{1}{2}}D''\alpha\|^{2}_{D_v\setminus S,h'}\nonumber\\
&\quad\geq\langle\langle[\eta\sqrt{-1}\Theta_{h'}-\sqrt{-1}\partial\bar\partial\eta-
\sqrt{-1}g\partial\eta\wedge\bar\partial\eta,\Lambda_{\omega}]
\alpha,\alpha\rangle\rangle_{D_v\setminus S,h'}\nonumber\\
&\quad\geq\bigg\langle\bigg\langle\bigg[\eta\sqrt{-1}\partial\bar\partial\phi+\frac{1}{\delta}\sqrt{-1}
\Theta_{h{\rm e}^{-\Psi}}-\sqrt{-1}\partial\bar\partial\eta-
\sqrt{-1}g\partial\eta\wedge\bar\partial\eta,\Lambda_{\omega}\bigg]
\alpha,\alpha\bigg\rangle\bigg\rangle_{D_v\setminus S,h'},
\end{align}
where $g$ is a positive continuous function on $D_{v}\setminus S$.
We need some calculations to determine $g$.

We have
\begin{equation}
\label{}
\partial\bar{\partial}\eta=-s'(-v_{t_0,\varepsilon}\circ \Psi)
\partial\bar{\partial}(v_{t_0,\varepsilon}\circ \Psi)
+s''(-v_{t_0,\varepsilon}\circ
\Psi)\partial(v_{t_0,\varepsilon}\circ \Psi)\wedge
\bar{\partial}(v_{t_0,\varepsilon}\circ \Psi),
\end{equation}
and
\begin{equation}
\label{}
\partial\bar{\partial}\phi=-u'(-v_{t_0,\varepsilon}\circ \Psi)
\partial\bar{\partial}v_{t_0,\varepsilon}\circ \Psi
+ u''(-v_{t_0,\varepsilon}\circ
\Psi)\partial(v_{t_0,\varepsilon}\circ \Psi)
\wedge\bar{\partial}(v_{t_0,\varepsilon}\circ \Psi).
\end{equation}

Therefore
\begin{align}
\label{equ:vector1}
&\eta\sqrt{-1}\partial\bar\partial\phi-\sqrt{-1}\partial\bar\partial\eta-
\sqrt{-1}g\partial\eta\wedge\bar\partial\eta\nonumber\\
&\quad=(s'-su')\sqrt{-1}\partial\bar{\partial}(v_{t_0,\varepsilon}\circ \Psi)
+((u''s-s'')-gs'^{2})\sqrt{-1}\partial(v_{t_0,\varepsilon}\circ
\Psi)\wedge\bar {\partial}(v_{t_0,\varepsilon}\circ \Psi)\nonumber\\
&\quad=
(s'-su')((v'_{t_0,\varepsilon}\circ\Psi)\sqrt{-1}\partial\bar{\partial}
\Psi+(v''_{t_0,\varepsilon}\circ
\Psi)\sqrt{-1}\partial(\Psi)\wedge\bar{\partial}(\Psi))\nonumber\\
&\qquad+((u''s-s'')-gs'^{2})\sqrt{-1}\partial(v_{t_0,\varepsilon}\circ \Psi)
\wedge\bar{\partial}(v_{t_0,\varepsilon}\circ \Psi).
\end{align}

We omit composite item $(-v_{t_0,\varepsilon}\circ \Psi)$ after
$s'-su'$ and $(u''s-s'')-gs'^{2}$ in the above equalities.

Denote $g=\frac{u''s-s''}{s'^{2}}\circ(-v_{t_0,\varepsilon}\circ
\Psi)$. We have
$\eta+g^{-1}=(s+\frac{s'^{2}}{u''s-s''})\circ(-v_{t_0,\varepsilon}\circ
\Psi)$.
Since $\Theta_{h{\rm e}^{-\Psi}}\geq0$ and
$\Theta_{h{\rm e}^{-(1+\delta)\Psi}}\geq0$ on $M\setminus S$ and $0\leq
v'_{t_{0},\varepsilon}\circ\Psi\leq1$, we have
\begin{equation}
(1-v'_{t_0,\varepsilon}\circ\Psi)\Theta_{h{\rm e}^{-\Psi}}+(v'_{t_0,\varepsilon}
\circ\Psi)\Theta_{h{\rm e}^{-(1+\delta)\Psi}}\geq 0
\end{equation}
on $M\setminus S$, which means
\begin{equation}
\label{equ:vector2}
\frac{1}{\delta}\Theta_{h{\rm e}^{-\Psi}}+(v'_{t_0,\varepsilon}\circ\Psi)\partial
\bar{\partial}\Psi\geq 0
\end{equation}
on $M\setminus S$.

As $v'_{t_0,\varepsilon}\geq 0$  and $s'-su'=1$, by equalities
(\ref{equ:10.1}), (\ref{equ:vector1}) and inequality (\ref{equ:vector2}),
we have
\begin{align}
\label{equ:vector3}
\langle B\alpha, \alpha\rangle_{h'}&=\langle[\eta\sqrt{-1}
\Theta_{h'}-\sqrt{-1}\partial\bar\partial\eta-\sqrt{-1}g
\partial\eta\wedge\bar\partial\eta,\Lambda_{\omega}]
\alpha,\alpha\rangle_{h'}\nonumber\\
&\geq
\langle[(v''_{t_0,\varepsilon}\circ \Psi)\sqrt{-1}\partial\Psi\wedge
\bar{\partial}\Psi,\Lambda_{\omega}]\alpha,\alpha\rangle_{h'}
=\langle (v''_{t_{0},\varepsilon}\circ \Psi) \bar\partial\Psi\wedge
(\alpha\llcorner(\bar\partial\Psi)^\sharp ),\alpha\rangle_{h'}.
\end{align}

By the definition of contraction, Cauchy-Schwarz inequality and the
inequality (\ref{equ:vector3}), we have
\begin{align}
\label{}
|\langle (v''_{t_{0},\varepsilon}\circ \Psi)\bar\partial\Psi\wedge
u,v\rangle_{h'}|^{2} &=|\langle (v''_{t_{0},\varepsilon}\circ \Psi)
u,v\llcorner(\bar\partial\Psi)^ \sharp \rangle_{h'}|^{2}\nonumber
\\&\leq\langle( v''_{t_{0},\varepsilon}\circ \Psi) u,u\rangle_{h'}
(v''_{t_{0},\varepsilon}\circ
\Psi)|v\llcorner(\bar\partial\Psi)^\sharp|_{h'}^{2}\nonumber
\\&=\langle (v''_{t_{0},\varepsilon}\circ \Psi) u,u\rangle_{h'}
\langle (v''_{t_{0},\varepsilon}\circ \Psi) \bar\partial\Psi\wedge
(v\llcorner (\bar\partial\Psi)^\sharp ),v\rangle_{h'}\nonumber
\\&\leq\langle (v''_{t_{0},\varepsilon}\circ \Psi )u,u\rangle_{h'}
\langle Bv,v\rangle_{h'},
\end{align}
for any $(n,q)$-form $u$ and $(n,q+1)$-form $v$.

Let
$\lambda=\bar{\partial}[(1-v'_{t_0,\varepsilon}(\Psi)){\tilde{F}}]$,
$u=\tilde{F}$, and $v=B^{-1}\bar\partial\Psi\wedge \tilde{F}$, we
have
$$\langle B^{-1}\lambda,\lambda\rangle_{h'} \leq (v''_{t_0,\varepsilon}
\circ{\Psi})| \tilde{F}|^2_{h'}.$$
Then it is clear that
 $$\int_{D_v\setminus S}\langle B^{-1}\lambda,\lambda\rangle_{h'} dV_{M}
 \leq \int_{D_v\setminus S}(v''_{t_0,\varepsilon}\circ{\Psi})| \tilde{F}|^2_{h'}dV_{M}.$$

By Lemma \ref{l:vector7}, we get $u_{v,t_0,\varepsilon}$ on
$D_{v}\setminus S$ which is an $(n,0)$-form with values in $E$
satisfying $\bar{\partial}u_{v,t_0,\varepsilon}=\lambda$, such that
\begin{equation}
 \label{equ:vector3.2}
 \int_{ D_v\setminus S}|u_{v,t_0,\varepsilon}|^{2}_{h'}(\eta+g^{-1})^{-1}dV_{M}
  \leq\int_{D_v\setminus S}(v''_{t_0,\varepsilon}\circ{\Psi})| \tilde{F}|^2_{h'}dV_M.
\end{equation}

Denote $\mu_{1}={\rm e}^{v_{t_0,\varepsilon}\circ\Psi}$ and
$\mu=\mu_{1}{\rm e}^{\phi}$. Assume one can choose $\eta$ and $\phi$ such
that $\mu\leq C(\eta+g^{-1})^{-1}$, where $C$ is
just the constant in Theorem \ref{t:guan-zhou-vector}.

Note that $v_{t_0,\varepsilon}(\Psi)\geq\Psi$, then we obtain
\begin{equation}
\label{equ:vector3.8}
\int_{ D_v\setminus S}|u_{v,t_0,\varepsilon}|^{2}_{h}dV_{M}
\leq\int_{ D_v\setminus
S}|u_{v,t_0,\varepsilon}|^{2}_{h'}\mu_{1}{\rm e}^{\phi} dV_{M}.
\end{equation}

By inequalities (\ref{equ:vector3.2}) and (\ref{equ:vector3.8}), we have
$$\int_{D_v\setminus S}|u_{v,t_0\varepsilon}|^{2}_{h}dV_{M}
\leq C\int_{D_v\setminus S}
(v''_{t_0,\varepsilon}\circ{\Psi})| \tilde{F}|^2_{h'}dV_M,$$ under
the assumption $\mu\leq C (\eta+g^{-1})^{-1}$.

For any given $t_{0}$ there exist $m_{0}$ and a neighborhood
$U_{0}$ of $\{\Psi=-\infty\}\cap \overline{D_{v}}$ on $M$, such that
for any $\varepsilon$,
$v''_{t_0,\varepsilon}\circ\Psi_{v,m}|_{U_{0}}=0$, we have
$\bar\partial u_{v,t_0,\varepsilon}|_{U_0\setminus S}=0$.

Note that $u_{v,t_0,\varepsilon}$ is locally $L^{2}$ integrable
along $S$, we have $u_{v,t_0,\varepsilon}$ can be extended to
$U_{0}$ as a holomorphic function, which is denoted by
$\tilde{u}_{v,t_0,\varepsilon}$.

As $\Psi\in \Delta_{h,\delta}$, we see that ${\rm e}^{-\Psi}$ is
disintegrable near $S$. Then it is clear that
$\tilde{u}_{v,t_0,\varepsilon}$ satisfies
$\tilde{u}_{v,t_0,\varepsilon}|_{S}=0$, and
\begin{equation}
\label{equ:vector3.3}\int_{
D_v}|\tilde{u}_{v,t_0,\varepsilon}|^{2}_{h}dV_{M}
\leq\frac{C}{{\rm e}^{A_{t_0}}}\int_{D_v}
(v''_{t_0,\varepsilon}\circ{\Psi_{v}})|
\tilde{F}|^2_{h{\rm e}^{-\Psi}}dV_M,
\end{equation}
where $A_{t_0}:=\sup_{t\geq t_0}\{u(t)\}$.

As $\lim_{t\to+\infty}u(t)=-\log(1+\frac{1}{\delta})$, it is clear
that
$\lim_{t_{0}\to\infty}\frac{1}{{\rm e}^{A_{t_0}}}=1+\frac{1}{\delta}$.

Let $F_{v,t_0,\varepsilon}:=(1-v'_{t_0,\varepsilon}\circ\Psi_{v})
\widetilde{F}-\tilde{u}_{v,t_0,\varepsilon}$. By
$\tilde{u}_{v,t_0,\varepsilon}|_{S}=0$, we have that
$F_{v,t_0,\varepsilon}$ is a holomorphic $(n,0)$-form on $D_{v}$
satisfying $F_{v,t_0,\varepsilon}|_{S}=\tilde{F}|_{S}$ and
inequality (\ref{equ:vector3.3}) is reformulated as follows:
\begin{equation}
\label{equ:vector3.5}
\int_{D_v}|F_{v,t_{0},\varepsilon}-(1-v'_{t_{0},\varepsilon}\circ\Psi)
\tilde{F}|^{2}_{h}dV_{M}
\leq\frac{C}{{\rm e}^{A_{t_0}}}\int_{D_v}(v''_{t_0,\varepsilon}
\circ\Psi_{v})|\tilde{F}|^{2}_{h{\rm e}^{-\Psi_{v}}}dV_{M}.
\end{equation}

Given $t_0$ and $D_{v}$, it is clear that
$(v''_{t_0,\varepsilon}\circ\Psi_{v})|\tilde{F}|^{2}_{h{\rm e}^{-\Psi_{v}}}$
have uniform bound on $D_{v}$ independent of $\varepsilon$.

Then
$\int_{D_v}|(1-v'_{t_0,\varepsilon}\circ\Psi_{v})\tilde{F}|^{2}_{h}dV_{M}$
and
$\int_{D_v}v''_{t_0,\varepsilon}\circ\Psi_{v}|\tilde{F}|^{2}_{h{\rm e}^{-\Psi_{v}}}dV_{M}$
have uniform bound independent of $\varepsilon$, for any given $t_0$
and $D_v$.

By $\bar\partial F_{v,t_{0},\varepsilon}=0$ and weakly compactness
of the unit ball of $L^{2}_{\varphi}(D_{v})$, we see that the weak
limit of some weakly convergent subsequence of
$\{F_{v,t_0,\varepsilon}\}_{\varepsilon}$ when $\varepsilon\to 0$
gives us a holomorphic $(n,0)$-form with values in $E$, which is
denoted by $F_{v,t_0}$ on $D_{v}$ and satisfies
$F_{v,t_0}|_{S}=\tilde{F}|_{S}$.

Note that we can also choose a subsequence of the weakly convergent
subsequence of $\{F_{v,t_0,\varepsilon}\}_{\varepsilon}$, such that
the chosen sequence is uniformly convergent on any compact subset of
$D_v$, denoted again by $\{F_{v,t_0,\varepsilon}\}_{\varepsilon}$
without ambiguity.

For any compact subset $K$ on $D_v$, it is clear that
$F_{v,t_0,\varepsilon}$,
$(1-v'_{t_0,\varepsilon}\circ\Psi_{v})\tilde{F}$ and
$(v''_{t_0,\varepsilon}\circ\Psi_{v})|\tilde{F}|^{2}{\rm e}^{-\varphi-\Psi_{v}}$
have uniform bounds on $K$ independent of $\varepsilon$.

By using dominated convergence theorem on any compact subset $K$ of
$D_v$ and inequality (\ref{equ:vector3.5}), we have
\begin{equation}
\int_{K}|F_{v,t_0}-(1-b_{t_0}(\Psi_{v}))\tilde{F}|^{2}_{h}dV_{M}
\leq\frac{C}{{\rm e}^{A_{t_0}}}\int_{D_v}({1}_{\{-t_{0}-1< t<-t_{0}\}}
\circ\Psi_{v})|\tilde{F}|^{2}_{h{\rm e}^{-\Psi_{v}}}dV_{M},
\end{equation}
which implies
\begin{equation}
\label{equ:vector3.4}
\int_{ D_v}|F_{v,t_0}-(1-b_{t_0}(\Psi_{v}))\tilde{F}|^{2}_{h}dV_{M}
\leq\frac{{C}}{{\rm e}^{A_{t_0}}}\int_{D_v}({1}_{\{-t_{0}-1< t<-t_{0}\}}
\circ\Psi_{v})|\tilde{F}|^{2}_{h{\rm e}^{-\Psi_{v}}}dV_{M}.
\end{equation}

By the definition of $dV_{M}[\Psi]$ and since
$\sum_{k=1}^{n}\frac{\pi^{k}}{k!}\int_{S_{n-k}}|f|^{2}_{h}dV_{M}[\Psi]<\infty$,
we have
\begin{align}
\label{equ:vector3.6}
&\limsup_{t_{0}\to+\infty}\int_{D_v}({1}_{\{-t_{0}-1<
t<-t_{0}\}} \circ\Psi_{v})|\tilde{F}|^{2}_{h{\rm e}^{-\Psi_{v}}}dV_{M}\nonumber\\
&\quad\leq
\limsup_{t_{0}\to+\infty}\int_{M}{1}_{\overline{D}_{v}}
({1}_{\{-t_{0}-1<t<-t_{0}\}}\circ\Psi)|\tilde{F}|^{2}_{h{\rm e}^{-\Psi_{v}}}dV_{M}\nonumber\\
&\quad\leq\sum_{k=1}^{n}\frac{\pi^{k}}{k!}\int_{S_{n-k}}
{1}_{\overline{D}_{v}}|f|^{2}_{h}dV_{M}[\Psi]
\leq\sum_{k=1}^{n}\frac{\pi^{k}}{k!}\int_{S_{n-k}}|f|^{2}_{h}dV_{M}[\Psi]<\infty.
\end{align}

Then $\int_{D_v}({1}_{\{-t_{0}-1< t<-t_{0}\}}
\circ\Psi_{v})|\tilde{F}|^{2}_{h{\rm e}^{-\Psi_{v}}}dV_{M}$ have uniform
bound independent of $t_{0}$, for any given $D_v$, and
\begin{equation}
\label{equ:vector3.7}
\limsup_{t_{0}\to+\infty}\int_{D_v}({1}_{\{-t_{0}-1<
t<-t_{0}\}} \circ\Psi_{v})|\tilde{F}|^{2}_{h{\rm e}^{-\Psi_{v}}}dV_{M}
\leq\sum_{k=1}^{n}\frac{\pi^{k}}{k!}\int_{S_{n-k}}|f|^{2}_{h}dV_{M}[\Psi]<\infty.
\end{equation}

It is clear that $\int_{
D_v}|F_{v,t_0}-(1-b_{t_0}(\Psi_{v}))\tilde{F}|^{2}_{h}dV_{M}$ have
uniform bound independent of $t_{0}$, for any given $D_v$.

As $\int_{ D_v}|(1-b_{t_0}(\Psi_{v}))\tilde{F}|^{2}_{h}dV_{M}$ have
uniform bound independent of $t_{0}$, by inequality
(\ref{equ:vector3.4}) and
\begin{equation}
\label{equ:vector3.9}
\bigg(\int_{
D_v}|F_{v,t_0}-(1-b_{t_0}(\Psi_{v}))\tilde{F}|^{2}_{h}dV_{M}\bigg)^
{\frac{1}{2}}+\bigg(\int_{
D_v}|(1-b_{t_0}(\Psi_{v}))\tilde{F}|^{2}_{h}dV_{M}\bigg)^
{\frac{1}{2}}\geq \bigg(\int_{
D_v}|F_{v,t_0}|^{2}_{h}dV_{M}\bigg)^{\frac{1}{2}},
\end{equation}
we can obtain that $\int_{ D_v}|F_{v,t_0}|^{2}_{h}dV_{M}$ have
uniform bound independent of $t_{0}$.

Because of $\bar\partial F_{v,t_{0}}=0$ and weakly compactness of
the unit ball of $L^{2}_{\varphi}(D_{v})$, we see that the weak
limit of some weakly convergent subsequence of $\{F_{v,t_0}\}_{t_0}$
when $t_{0}\to+\infty$ gives us a holomorphic $(n,0)$-form with
values in $E$, which is denoted by $F_{v}$ on $D_{v}$ and satisfies
$F_{v}|_{S}=\tilde{F}|_{S}$.

Note that we can also choose a subsequence of the weakly convergent
subsequence of $\{F_{v,t_{0}}\}_{t_{0}}$, such that the chosen
sequence is uniformly convergent on any compact subset of $D_v$,
denoted by $\{F_{v,t_0}\}_{t_{0}}$ without ambiguity.

For any compact subset $K$ on $D_v$, it is clear that both of
$F_{v,t_0}$ and $(1-b_{t_0}\circ\Psi_{v})\tilde{F}$ have uniform
bound on $K$ independent of $t_0$.

By inequalities (\ref{equ:vector3.4}), (\ref{equ:vector3.7})
and dominated convergence theorem on any compact subset $K$ of
$D_v$, we have
\begin{equation}
\int_{D_v}{1}_{K}|F_{v}|^{2}_{h}dV_{M}
\leq\frac{{C}}{{\rm e}^{A_{t_0}}}\sum_{k=1}^{n}\frac{\pi^{k}}{k!}
\int_{S_{n-k}}|f|^{2}_{h}dV_{M}[\Psi],
\end{equation}
which implies
\begin{equation}
\int_{ D_v}|F_{v}|^{2}_{h}dV_{M}
\leq\frac{{C}}{{\rm e}^{A_{t_0}}}\sum_{k=1}^{n}\frac{\pi^{k}}{k!}
\int_{S_{n-k}}|f|^{2}_{h}dV_{M}[\Psi].
\end{equation}
Note that the Lebesgue measure of $\{\Psi=-\infty\}$ is zero.

It suffices to find $\eta$ and $\phi$ such that $$(\eta+g^{-1})\leq
C{\rm e}^{-\Psi_{v}}{\rm e}^{-\phi}=C\mu^{-1}\quad {\rm on}\ D_v.$$ As
$\eta=s(-v_{t_0,\varepsilon}\circ\Psi_{v})$ and
$\phi=u(-v_{t_0,\varepsilon} \circ\Psi_{v})$, we have $$(\eta+g^{-1})
{\rm e}^{v_{t_0,\varepsilon}
\circ\Psi_{v}}{\rm e}^{\phi}=(s+\frac{s'^{2}}{u''s-s''}){\rm e}^{-t}{\rm e}^{u}
\circ(-v_{t_0,\varepsilon}\circ\Psi_{v}).$$

We are naturally led to obtain the following system of ODEs:
\begin{equation}
\label{vectorGZ}
\begin{split}
&(1)\ \bigg(s+\frac{s'^{2}}{u''s-s''}\bigg){\rm e}^{u-t}=C, \\
&(2)\ s'-su'=1,
\end{split}
\end{equation}
where $t\in[0,+\infty)$, and $C=1$.

One can solve the ODEs by the same method as in Remark \ref{GZ} and
get $$u=-\log\bigg(1+\frac{1}{\delta}-{\rm e}^{-t}\bigg)\quad {\rm and}\quad
s=\frac{(1+\frac{1}{\delta})t+\frac{1}{\delta}(1+\frac{1}{\delta})}
{1+\frac{1}{\delta}-{\rm e}^{-t}}-1,$$ which satisfy the ODEs
(\ref{vectorGZ}).

One may check that $s\in C^{\infty}((0,+\infty))$ satisfies
$s\geq\frac{1}{\delta}$,
$\lim_{t\to+\infty}u(t)=-\log(1+\frac{1}{\delta})$ and $u\in
C^{\infty}((0,+\infty))$ satisfies $u''s-s''>0$.

Define $F_v=0$ on $M\backslash D_v$. As $\lim_{t_{0}\to
\infty}A_{t_{0}}=-\log(1+\frac{1}{\delta})$, then the weak limit of
some weakly convergent subsequence of $\{F_v\}_{v=1}^\infty$ gives
us a holomorphic section $F$ of $K_{M}\otimes E$ on $M$ satisfying
$F|_{S}=\tilde{F}|_{S}$, and
$$\int_{ M}|F|^{2}_{h}dV_{M}
\leq C\bigg(1+\frac{1}{\delta}\bigg)\sum_{k=1}^{n}\frac{\pi^{k}}{k!}
\int_{S_{n-k}}|f|^{2}_{h}dV_{M}[\Psi],$$ where
$c_{k}=(-1)^{\frac{k(k-1)}{2}}\sqrt{-1}^k$ for $k\in\mathbb{Z}$.

In conclusion, we have proved Theorem \ref{t:guan-zhou-vector} with
the constant $C=1$.

\section{Bergman kernel and logarithmic capacity on Riemann
surfaces}
In this section, we show some relationships between Bergman kernel
and logarithmic capacity on compact and open Riemann surfaces.

\subsection{Proof of Theorem \ref{t:arak}}
Let ${\rm e}^{2\varphi}|dz|^{2}$ be the Poincar\'{e} metric on $X$, it is
clear that
 $\omega={\rm e}^{2\varphi}dz\wedge d\bar{z}$.
Then ${\rm e}^{-2\varphi}$ is a Hermitian metric of $K_{X}$ on $X$, denoted by $h$.
It is known that
$$c_{1}(K_{X})=\frac{\sqrt{-1}}{\pi}\partial\bar\partial(2\varphi)=b\omega,$$
where $b$ is a positive constant, and
$$\frac{\sqrt{-1}}{\pi}\partial\bar\partial g(\cdot,q)|_{X\setminus q}=-a\omega,$$
where $a$ is a positive constant.

Note that $$\int_{X}\frac{\sqrt{-1}}{\pi}\partial\bar\partial
g(\cdot,q)=0\quad {\rm and}\quad
\frac{\sqrt{-1}}{\pi}\partial\bar\partial g(\cdot,q)=[\{q\}]-a\omega,$$
then we have
$$\int_{X}a\omega=1.$$
Note that $$\int_{X}b\omega=\int_{X}c_{1}(K_{X})=2(g-1),$$
then it is clear that
$$\frac{2a}{b}=\frac{1}{g-1}.$$
Let $\Psi=2g(\cdot,p)$. We choose
$$\delta=(g-1)(m-1)-1,$$ then it is clear that
$$\Psi\in\Delta_{h^{m},\delta}\quad {\rm and}\quad
1+\frac{1}{\delta}=1+\frac{1}{(g-1)(m-1)-1}.$$
Without loss of generality, we can assume that $p$ is $o\in\Delta$,
$z\in\Delta$ is the local coordinate near $p$. Then by the
definition of $c_{X}(p)$, the relation between Euclidean distance
and Poincar\'{e} distance near $o\in\Delta$, and Lemma \ref{l:lem9},
we have
\begin{align}
\label{equ:arak1}
\omega[\Psi]&={\rm e}^{2\varphi}dz\wedge d\bar{z}[2\log {\rm dist}_{\omega}(\cdot,p)+2\log c_{X}(p)]\nonumber
\\&={\rm e}^{2\varphi}dz\wedge d\bar{z}[2\log {\rm dist}_{\omega}(\cdot,p)]c^{-2}_{X}(p)\nonumber
\\&=2{\rm e}^{2\varphi}d\lambda_{z}[2\log |z-p|+2\varphi(p)]c^{-2}_{X}(p)\nonumber
\\&=2{\rm e}^{2\varphi}d\lambda_{z}[2\log |z-p|]{\rm e}^{-2\varphi(p)}c^{-2}_{X}(p)\nonumber
\\&=2d\lambda_{z}[2\log |z-p|]c^{-2}_{X}(p)\nonumber
\\&=2[\{p\}]c^{-2}_{X}(p),
\end{align}
where $d\lambda_{z}$ is the Lebesgue measure with respect to $z$.
Corollary \ref{coro:line1} and equality (\ref{equ:arak1}) tell us that there exists a
holomorphic section $F$ of $mK_{X}$ on
$X$, such that
$F|_{p}=(dw)^{m}$ and
$$\int_{X}\sqrt{-1}|F|^{2}_{h^{m}}\omega\leq\pi\bigg(1+\frac{1}{\delta}\bigg)2|
(dw)^{m}|^{2}_{h^{m}}c_{X}^{-2}(p).$$

Since
$$|\kappa_{X}(p,p)|_{h^{m}}\geq \frac{2|f|^{2}_{h^{m}}|_{p}}{\int_{X}|f|^{2}_{h^{m}
}\omega},$$ for any nonzero holomorphic section $f$ of $mK_{X}$ on
$X$, we obtain the theorem.


\subsection{Proof of Theorem \ref{t:arak1}}
As $X$ is a complex torus, it is known that $c_{1}(L)=b\omega$, $b$
is a positive constant. It is known that
$$\frac{\sqrt{-1}}{\pi}\partial\bar\partial g(\cdot,q)|_{X\setminus q}=-a\omega,$$
where $a$ is a positive constant.

Note that $$\int_{X}\frac{\sqrt{-1}}{\pi}\partial\bar\partial
g(\cdot,q)=0\quad{\rm and}\quad
\frac{\sqrt{-1}}{\pi}\partial\bar\partial g(\cdot,q)=[\{q\}]-a\omega,$$
then we have
$$\int_{X}a\omega=1.$$

Note that $$\int_{X}b\omega=\int_{X}c_{1}(L)=d,$$ then it is clear
that
$$\frac{2a}{b}=\frac{2}{d}.$$
Let $\Psi=2g(\cdot,p)$. We choose
$$\delta=\frac{d}{2}-1,$$ then it is clear that
$$\Psi\in\Delta_{h_{L},\delta}\quad{\rm and}\quad
1+\frac{1}{\delta}=1+\frac{1}{\frac{d}{2}-1}.$$

By the same method as in equality (\ref{equ:arak1}), we have
\begin{equation}
\omega[\Psi]=2[\{p\}]c^{-2}_{X}(p).
\end{equation}

Corollary \ref{coro:line1} and equality (\ref{equ:arak1}) tell us that
there exists a holomorphic section $F$ of $K_{X}\otimes L$ on $X$,
such that $F|_{p}=f(0)$ and
$$\int_{X}\sqrt{-1}|F|^{2}_{h_{L}}\omega\leq\pi\bigg(1+\frac{1}{\delta}\bigg)2|f(0)|^{2}
_{h_{L}}c_{X}^{-2}(p).$$

Since
$$|\kappa_{X,d}(p,p)|_{h_{L}}\geq \frac{2|f|^{2}_{h_{L}}|_{p}}{\int_{X}|f|^{2}
_{h_{L}}\omega},$$ for any nonzero holomorphic section $f$ of
$K_{X}\otimes L$ on $X$, we obtain the theorem.




\subsection{Proof of Corollary \ref{c:extended_suita}}
Note that $2G_{\Omega}(\xi,z)-\log|\xi-z|^{2}$ is smooth with respect
to $\xi$, on the coordinate neighborhood of $z$, when $\Omega$ is a
Riemann surface admitting the Green function $G_{\Omega}(\xi,z)$.

By Lemma \ref{l:lem9}, we have
\begin{equation}
\label{equ:sui1.1}
d\lambda_{z}[2G_{\Omega}(\xi,z)]=d\lambda_{z}[\log|\xi-z|^{2}+(2G_{\Omega}
(\xi,z)-\log|\xi-z|^{2})]={\rm e}^{-(2G_{\Omega}(\xi,z)-\log|\xi-z|^{2})}\delta_{z}.
\end{equation}

Let $w$ be the local coordinate of a neighborhood of $z_{0}$. Then extended Suita
conjecture becomes
$$(c_{\beta}(z_{0}))^{2}|dw|^{2}\leq\pi \rho(z_{0})\kappa_{\Omega,\rho}(z_{0}).$$

Let $M=\Omega$, $\varphi=2h$, $S=z_0$, and
$\Psi(z)=2G_{\Omega}(z,z_{0})$ in Theorem \ref{t:guanzhou1}. Then
the theorem and equality~(\ref{equ:sui1.1}) tell us that there exists
a holomorphic $(1,0)$-form $F$ on $\Omega$, such that $F|_{z_{0}}=dw$
and
$$\int_{\Omega}\sqrt{-1}\rho F\wedge\bar{F}\leq\frac{2{C}\pi\rho(z_{0})}
{(c_{\beta}(z_{0}))^{2}}.$$

Since
$$\kappa_{\Omega}(z_{0})\geq \frac{2f\otimes\bar{f}|_{z_{0}}}{\sqrt{-1}
\int_{\Omega}\rho f\wedge\bar{f}}$$ for any nonzero holomorphic
$(1,0)$-form $f$ on $\Omega$, we obtain Corollary
\ref{c:extended_suita}.

\subsection{Proof of Proposition \ref{t:bound}}
As
$C(\Omega,z_{0}):=\frac{c_{\beta}(z_{0})^{2}|dz|^{2}}{\pi\kappa_
{\Omega}(z_{0})}$, by arguments in the proof of Corollary
\ref{c:extended_suita}, it is clear that there exists a holomorphic
$(1,0)$-form $F$ on $\Omega$ satisfying
\begin{equation}
\label{equ:bound1}
\int_{\Omega}\sqrt{-1}F\wedge\bar{F}={C}(\Omega,z_{0})\pi\int_{z_0}|F|^{2}dV_{\Omega}
[2G_{\Omega}(\cdot,z_{0})],
\end{equation}
and $F|_{z_{0}}\neq 0$.

By Theorem 2.1 in \cite{dgz}, there exists a holomorphic embedding
$f: \Omega\rightarrow \Delta$ such that $f(z_{0})=0$ and $\Delta(0,
s_\Omega(z_{0}))\subset f(\Omega)$.

By the submean value inequality of
plurisubharmonic function on $\Delta(0, s_\Omega(z_{0}))$, we have
$$\int_{f(\Omega)}\sqrt{-1}f_{*}F\wedge f_{*}\bar{F}\geq \pi s_{\Omega}^{2}(z_{0})
\int_{f(z_0)}|f_{*}F|^{2}dV_{f(\Omega)}[\log|w|^{2}].$$

As $\int_{f(\Omega)}\sqrt{-1}f_{*}F\wedge f_{*}\bar{F}=
\int_{\Omega}\sqrt{-1}F\wedge\bar{F},$ it is clear that
$$\int_{\Omega}\sqrt{-1}F\wedge\bar{F}\geq \pi s_{\Omega}^{2}(z_{0})
\int_{f(z_0)}|f_{*}F|^{2}dV_{f(\Omega)}[\log|w|^{2}],$$ where $w$ is
the coordinate of $\Delta\subset\mathbb{C}$.

By equality (\ref{equ:bound1}), we obtain that
$$\pi{C}(\Omega,z_{0})\int_{z_0}|F|^{2}dV_{\Omega}[2G_{\Omega}(\cdot,z_{0})]
\geq \pi s_{\Omega}^{2}(z_{0})\int_{z_0}|f_{*}F|^{2}dV_{f(\Omega)}[\log|w|^{2}].$$
As $\log|w|^{2}-f_{*}2G_{\Omega}(\cdot,z_{0})$ is a negative smooth function on $\Omega$,
by Lemma \ref{l:lem9},
we have $$\int_{f(z_0)}|f_{*}F|^{2}dV_{f(\Omega)}[f_{*}2G_{\Omega}(\cdot,z_{0})]\leq
\int_{f(z_0)}|f_{*}F|^{2}dV_{f(\Omega)}[\log|w|^{2}].$$
Note that $$\int_{z_0}|F|^{2}dV_{\Omega}[2G_{\Omega}(\cdot,z_{0})]=
\int_{f(z_0)}|f_{*}F|^{2}dV_{f(\Omega)}[f_{*}2G_{\Omega}(\cdot,z_{0})],$$
and $1\geq{C}(\Omega,z_{0})$,
then we have ${C}(\Omega,z_{0})\geq s_{\Omega}^{2}(z_{0})$.

When $\lim_{z\to\partial\Omega}s_{\Omega}(z)=1$, it is clear that
$\lim_{z\to\partial\Omega}{C}(\Omega,z)=1$.

\Acknowledgements{This work was supported by National Natural Science Foundation of China (Grant No. 11031008).
The authors thank Prof. Ohsawa for explaining his work. The authors are grateful
to Prof. Yum-Tong Siu and Prof. Jean-Pierre Demailly for their
encouragements and useful discussions. An announcement of the
present paper appears in \cite{guan-zhou12a}.
}


\end{document}